\documentclass[12pt]{amsart}

\usepackage{latexsym}
\usepackage{psfrag}
\usepackage{amsmath}
\usepackage{amssymb}
\usepackage{epsfig}
\usepackage{amsfonts}
\usepackage{amscd}
\usepackage{mathrsfs}
\usepackage{graphicx}
\usepackage{enumerate}
\usepackage[autostyle=false, style=english]{csquotes}
\MakeOuterQuote{"}
\usepackage{ragged2e}
\usepackage[all]{xy}
\usepackage{mathtools}
\newlength\ubwidth

\usepackage[dvipsnames]{xcolor}

\usepackage{placeins}
\usepackage{tikz}
\usetikzlibrary{shapes, positioning}
\usetikzlibrary{matrix,shapes.geometric,calc,backgrounds}

\newtheorem{theorem}{Theorem}

\newtheorem{ex}[theorem]{Example}

\newtheorem{cor}[theorem]{Corollary}

\newtheorem{defn}[theorem]{Definition}

\usepackage{amsmath,amssymb,amsthm}
\usepackage[dvipsnames]{xcolor}
\usepackage[colorlinks=true,citecolor=RoyalBlue,linkcolor=red,breaklinks=true]{hyperref}

\usepackage{graphicx}
\usepackage{mathdots} 
\usepackage[margin=2.5cm]{geometry}
\usepackage{ytableau}

\newcommand{\partsize}{\mathrm{st}}

\begin{document}

\title[Part Sizes in Cylindric Partitions]
{Part Size Count in Cylindric Partitions and Bijections for Small Profiles}

\author[Kur\c{s}ung\"{o}z]{Ka\u{g}an Kur\c{s}ung\"{o}z}
\address{Ka\u{g}an Kur\c{s}ung\"{o}z, Faculty of Engineering and Natural Sciences,
Sabanc{\i} University, Tuzla, Istanbul 34956, Turkey}
\email{kursungoz@sabanciuniv.edu}

\author[\"{O}mr\"{u}uzun Seyrek]{Hal\.{ı}me \"{O}mr\"{u}uzun Seyrek}
\address{Hal\.{ı}me \"{O}mr\"{u}uzun Seyrek, Faculty of Applied Sciences, Biruni University, Zeytinburnu, Istanbul 34015, Turkey}
\email{hseyrek@biruni.edu.tr}



\date{Sep 2026}

\begin{abstract}
We first study the part size statistic in cylindric partitions, 
listing some specializations of a crude but general formula.  
Then, we establish bijections between ordinary partitions 
and cylindric partitions of small profiles with additional restrictions.  
Lastly, we extend earlier formulas for generating functions of 
cylindric partitions into distinct parts 
and functional equations of unrestricted cylindric partition generating functions, 
involving the part size statistic.  
\end{abstract}

\maketitle

\section{Introduction and statement of results}
\label{secIntro}

Cylindric partitions were introduced
by Gessel and Krattenthaler~\cite{GesselKrattenthaler}.

\begin{defn}\label{def:cylin} Let $r$ and $\ell$ be positive integers.
Let $c=(c_1,c_2,\dots, c_r)$ be a composition, where $c_1+c_2+\dots+c_r=\ell$.
A \emph{cylindric partition with profile $c$} is a vector partition
$\Lambda = (\lambda^{(1)},\lambda^{(2)},\dots,\lambda^{(r)})$,
where each $\lambda^{(i)} = \lambda^{(i)}_1+\lambda^{(i)}_2 + \cdots +\lambda^{(i)}_{s_i}$ is a partition,
such that for all $i$ and $j$,
\[
\lambda^{(i)}_j\geq \lambda^{(i+1)}_{j+c_{i+1}} \quad \text{and} \quad \lambda^{(r)}_{j}\geq\lambda^{(1)}_{j+c_1}.
\]
The integers $r$ and $\ell$ are called
the \emph{rank}, and the \emph{level}, of the cylindric partition,
respectively.
\end{defn}

For example, the sequence
$\Lambda=((12, 6, 2),(10,6,2),(13,10,5,2))$
is a cylindric partition with profile $(1,0,2)$.
It has rank $r=3$ for the vector partition $\Lambda$ has three partitions in it,
and level $\ell = 3$ for the profile is a composition of three.
One can check that for all $j$, $\lambda^{(1)}_j\ge \lambda^{(2)}_{j}$, $\lambda^{(2)}_j\ge \lambda^{(3)}_{j+2}$
and $\lambda^{(3)}_j\ge \lambda^{(1)}_{j+1}$.
We can visualize the required inequalities by writing the partitions
in subsequent rows repeating the first row below the last one,
and shifting the rows below as much as necessary to the left.
Thus, the inequalities between parts
become the weakly decreasing of the parts to the right in each row,
and downward in each column.
\[
\begin{array}{ccc ccc ccc}
& & & 12 & 6 & 2 &\\
& & & 10 & 6 & 2 & \\
& 13 & 10 & 5 & 2 & \\
\textcolor{lightgray}{12} & \textcolor{lightgray}{6}
& \textcolor{lightgray}{2} & 
\end{array}
\]
The repeated first row is shown in gray.

The size $\vert \Lambda \vert$ of a cylindric partition
$\Lambda = (\lambda^{(1)},\lambda^{(2)},\dots,\lambda^{(r)})$
is defined to be the sum of all the parts in the partitions
$\lambda^{(1)},\lambda^{(2)},\dots,\lambda^{(r)}$.
The largest part of a cylindric partition $\Lambda$
is defined to be the maximum part among all the partitions in $\Lambda$,
and it is denoted by $\max(\Lambda)$.
The number of part sizes in a cylindric partition $\Lambda$ is
the count of distinct integers appearing in $\Lambda$,
and it is denoted by $\partsize(\Lambda)$, 
the initials are those of \emph{slice type}.  
For the cylindric partition $\Lambda$ given above,
$\vert \Lambda \vert = 68$, $\mathrm{max}(\Lambda) = 13$,
and $\partsize(\Lambda) = 6$.

The following generating function
\begin{align}
\nonumber
F_c(z,u,q)
:=\sum_{\Lambda\in \mathcal{P}_c}
z^{\max{(\Lambda)}} u^{\partsize(\Lambda)} q^{\vert \Lambda \vert}
\end{align}

is the generating function for cylindric partitions, where $\mathcal{P}_c$ denotes the set of all cylindric partitions with profile $c$.

In 2007, Borodin~\cite{Borodin}
showed that when one sets $u=z=1$ to this generating function, 
it turns out to be a very nice infinite product.

\begin{theorem}[Borodin, 2007]
\label{theorem-Borodin}
Let $r$ and $\ell$ be positive integers,
and let $c=(c_1,c_2,\dots,c_r)$ be a composition of $\ell$. Define $t:=r+\ell$
and $s(i,j) := c_i+c_{i+1}+\dots+ c_j$.
Then,
\begin{equation}
\label{BorodinProd}
F_c(1,1,q) = \frac{1}{(q^t;q^t)_\infty}
\prod_{i=1}^r \prod_{j=i}^r \prod_{m=1}^{c_i}
\frac{1}{(q^{m+j-i+s(i+1,j)};q^t)_\infty}
\prod_{i=2}^r \prod_{j=2}^i \prod_{m=1}^{c_i} \frac{1}{(q^{t-m+j-i-s(j,i-1)};q^t)_\infty}.
\end{equation}
\end{theorem}

The identity (\refeq{BorodinProd}) is a very strong tool
that gives product representations
of generating functions of cylindric partitions with a given profile explicitly.

Here and throughout,
we use the following standard $q$-Pochhammer symbols~\cite{GR}.
\begin{align}
\nonumber
(a; q)_n := \prod_{j = 1}^n (1 - aq^{j-1})
\quad \textrm{ and } \quad
(a; q)_\infty := \lim_{n \to \infty} (a; q)_n,
\end{align}
for any $n \in \mathbb{N}$, $a, q \in \mathbb{C}$,
and $\vert q \vert < 1$.
The last condition on $q$ will ensure that
all series and infinite products in this note converge absolutely~\cite{NT-Rama, GR}.

Cylindric partitions have been heavily studied in a multitude of aspects such as
decompositions~\cite{Borodin, Corteel-RR-RSK, CSV, 
Langer-I, Langer-II, Leeuwen, Tingley, Tingley-Correction},
various distinctness conditions~\cite{BU, K-OS, K-OS-2023},
evidently positive generating functions~\cite{ASW, CDU, CW, FFW,
KR-completeASW, KR-tight, Tsu, AU23, Warnaar-A2AG, Warnaar-BaileyTree},
applications of standard methods to restricted or unrestricted
cylindric partitions~\cite{Burcu, LU},
and connections to vertex operator algebras 
and their representations~\cite{ASW, CDU, KR-completeASW, KR-tight, 
Warnaar-A2AG, Warnaar-BaileyTree}.

In particular, in~\cite{K-OS} the authors considered horizontally slicing
cylindric partitions into skew diagrams
and distinguished some slices as \emph{pivot} slices
per the following definition.
The terminology concerning Young diagrams and skew diagrams is that of~\cite{Fulton}.

\begin{defn}[\cite{K-OS}]
\label{defPivot}
 Let $\beta_1 \backslash \beta_0$, $\beta_2 \backslash \beta_1$,
 \ldots, $\beta_s \backslash \beta_{s-1}$ be skew diagrams, for a positive integer $s$.
 Set $\beta_{s+1}$ be the partition obtained by adding 1 to each part of $\beta_s$.
 In other words, $\beta_j$'s are partitions for $j =$ 0, \ldots, $s+1$,
 and  the Young diagram of $\beta_j$ is contained in the Young diagram of $\beta_{j+1}$
 for $j = 0$, \ldots, $s$.
 Stack the Young diagrams of $\beta_0$, $\beta_1$, \ldots, $\beta_{s+1}$
 on top of each other such that the left and the top edges are aligned.
 For any $j = 1$, \ldots, $s$,
 if the rightmost column of $\beta_j \backslash \beta_{j-1}$
 is strictly to the right of the leftmost column of $\beta_{j+1} \backslash \beta_j$,
 then $\beta_j \backslash \beta_0$ is called a \emph{pivot}
 in the chain $\beta_0$, $\beta_1$, \ldots $\beta_s$.
\end{defn}

A detailed example is in Section \ref{secPrelim}.  
Slices are denoted by ${}_{\cdot}\Sigma$'s.
When emphasis is needed, the pivots are denoted by ${}_{\cdot}\Pi$'s.
The indices are written on the bottom left because of
necessity of excessive use of indices in cylindric partitions~\cite{GesselKrattenthaler}.  

We would like to draw attention to some of the findings. 
The first technical result of this paper is the formula below.  

\begin{theorem}
\label{thmCrudeGeneral}
For a chosen and fixed profile $c$,

\begin{align}
\label{eqCrudeGeneral}
& \sum_{\Lambda \in \mathcal{P}_c} q^{ \vert \Lambda \vert }
z^{\mathrm{max}(\Lambda)}
u^{ \partsize(\Lambda) }
= \frac{ \left( (1 - u)zq; q \right)_\infty }{ \left( zq; q \right)_\infty }
\left[ 1 + \sum_{m \geq 1} \sum_{ \substack{ \Pi_1 < \Pi_2 < \cdots < \Pi_m \\
		\textrm{a chain of } m \textrm{ pivots } } }
\prod_{j = 1}^m \frac{ u z q^{ \vert \Pi_j \vert } }{ 1 - (1 - u) z  q^{ \vert \Pi_j \vert } } \right], 
\end{align}
where the innermost sum on the right-hand side 
is over all possible chains of $m$ pivots.  
\end{theorem}

The applications include 
\begin{align}
\label{genFuncProfile11}
  & \sum_{\Lambda \in \mathcal{P}_{(1,1)}} q^{ \vert \Lambda \vert }
  z^{\mathrm{max}(\Lambda)}
  u^{ \partsize(\Lambda) }
= \frac{ \left( (1 - u)zq^2; q^2 \right)_\infty  \left( (1 - 2u)zq; q^2 \right)_\infty }
	{ \left( zq; q \right)_\infty  }, 
\end{align}
and 
\begin{align}
\label{genFuncP21}
  & \sum_{\Lambda \in \mathcal{P}_{(2,1)}} q^{ \vert \Lambda \vert }
  z^{\mathrm{max}(\Lambda)}
  u^{ \partsize(\Lambda) }
  = 1 + \sum_{m \geq 1} \sum_{ 0 < j_1 < j_2 < \cdots < j_m } 
  f_{2+r_1} f_{2+r_2} \ldots f_{2+r_s}
  \prod_{k = 1}^m \frac{ u z q^{ j_k } }{ 1 - z  q^{ j_k } }, 
\end{align}
where $r_1$ is the length of the maximal run of consecutive numbers 
in the sequence $j_1 < j_2 < \cdots < j_m$ starting with $j_1$, 
$r_2$ is the length of the maximal run of consecutive numbers 
in the sequence $j_1 < j_2 < \cdots < j_m$ starting with $j_{r_1 + 1}$, etc., 
and $f_i$'s are the Fibonacci numbers.  

We also describe bijections between ordinary integer partitions 
and cylindric partitions of small profiles with additional constraints.  
One example is the following identity.  

\begin{theorem}\label{c=(1,1)bijection}
Let $\mathcal C$ denote the set of cylindric partitions
$\Lambda=(\lambda^{(1)},\lambda^{(2)})$ of profile $(1,1)$, written as
pairs
\[
P_k=[\lambda^{(2)}_k,\lambda^{(1)}_k]=[b_k,a_k],
\qquad
1\leq k\leq N,
\]
where
\[
N=\max\bigl\{\ell(\lambda^{(1)}),\ell(\lambda^{(2)})\bigr\},
\]
with zeros appended to the shorter component when necessary, such that
$a_k-b_k\in\{0,1\}$ for every $k$, and such that, writing $J(\Lambda)=\{k:a_k=b_k+1\}$,
$S_k=\#\{j>k:j\in J(\Lambda)\}$, and $r_k=b_k-S_k$, there exists an integer $t$ with $0\leq t\leq N$ such that
\[
r_1>r_2>\cdots>r_t>0
\]
and
\[
r_k=0,\qquad t<k\leq N,
\]
where the first condition is vacuous when $t=0$.

Let $\mathcal P$ denote the set of pairs $(\mu,\beta)$ such that every
positive part occurring in $\mu$ has multiplicity exactly two, and
$\beta$ is a partition into distinct odd parts.

Then there is a weight-preserving bijection $\Phi:\mathcal
C\to\mathcal P$, $|\Lambda|=|\mu|+|\beta|$.
\end{theorem}

This paper is organized as follows.  
Section \ref{secPrelim} reviews the necessary definitions and terminology.  
Section \ref{secGeneral} proves Theorem \ref{thmCrudeGeneral} and some corollaries.  
Section \ref{secBijections} contains Theorem \ref{c=(1,1)bijection}, its proof, 
and its companions sorted from easy to difficult.  
Section \ref{secDist} is the updating of Section 4 in~\cite{K-OS} 
incorporating the new statistic part size count.
Section \ref{secFuncEqs} is a construction of a system of functional equations 
defining $F_c(z,u,q)$ for arbitrary but fixed profile $c$.  
Section \ref{secFuncEqs} is also an updating of (and a correction in) 
the functional equation (24) in~\cite{K-OS}.  

\section{Preliminaries}
\label{secPrelim}

This is the preliminaries section in~\cite{K-OS}
less proofs,
which in turn is mostly a reinterpretation
of the definitions in~\cite{GesselKrattenthaler}.

A visualization of cylindric partitions is obtained by
replacing each number in the cylindric partition by a vertical stack
consisting of that many unit cubes~\cite{Warnaar-A2AG},
then aligning and shifting as imposed by the profile.
The cylindric partition 	$\Lambda=((12, 6, 2),(10,6,2),(13,10,5,2))$
with profile $(1,0,2)$ will be shown as in Figure \ref{figStackOfCubes}.
The repeated first row is shown as faded.
\begin{figure}
\centering
\includegraphics[scale=0.3]{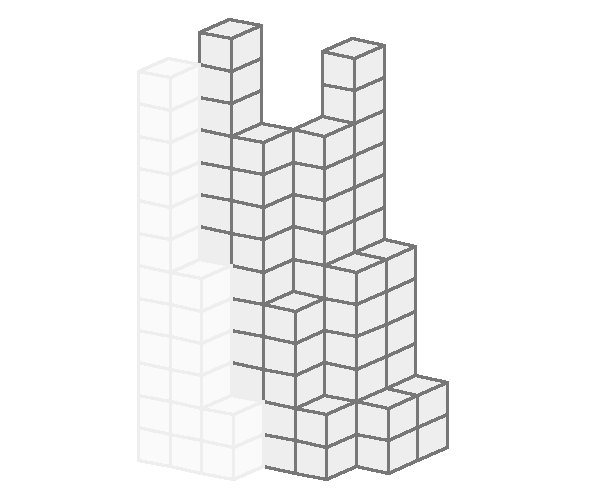}
\caption{ A visualization of a cylindric partition, the repeated first row is shaded. }
\label{figStackOfCubes}
\end{figure}

Given two cylindric partitions $\Lambda$ and $M$ having the same profile $c$,
we say $\Lambda \geq M$ if
$\lambda^{(i)}_j \geq \mu^{(i)}_j$ for all pairs $(i, j)$.
The missing entries may be filled in with zeros, as necessary to make comparisons.
This gives a partial ordering on cylindric partitions,
reminiscent of the partial order of integer partitions
by containment of the Young diagrams~\cite{TheBlueBook}.

It is also convenient to define the empty cylindric partition of a profile $c$
as an analog of the empty partition of zero in ordinary partitions.
For any fixed profile $c$,
the \emph{empty cylindric partition} $E$
is the unique partition satisfying $\Lambda \geq E$
for all cylindric partitions $\Lambda$ with profile $c$.

One defines $\Lambda+M$ as the componentwise addition,
taking the missing entries as zeros.
$\Lambda+M$ is another cylindric partition with the same profile,
and $\mathrm{max}(\Lambda+M) \leq \mathrm{max}(\Lambda) + \mathrm{max}(M)$.

Given any cylindric partition $\Lambda$,
it is uniquely possible to horizontally cut $\Lambda$ into \emph{slices}
${}_1\Sigma$, ${}_2\Sigma$, \ldots, ${}_m\Sigma$ such that
\begin{align}
\label{ineqSliceChain}
E \leq {}_1\Sigma \leq {}_2\Sigma \leq \cdots \leq {}_m\Sigma,
\end{align}
and
\begin{align}
\nonumber
\Lambda = {}_1\Sigma + {}_2\Sigma + \cdots + {}_m\Sigma,
\end{align}

The slices of our running example are as follows.  
\begin{center}
\raisebox{7mm}{\includegraphics[scale=0.3]{ex0.png} }
\raisebox{2cm}{$ \rightarrow $}
\includegraphics[scale=0.4]{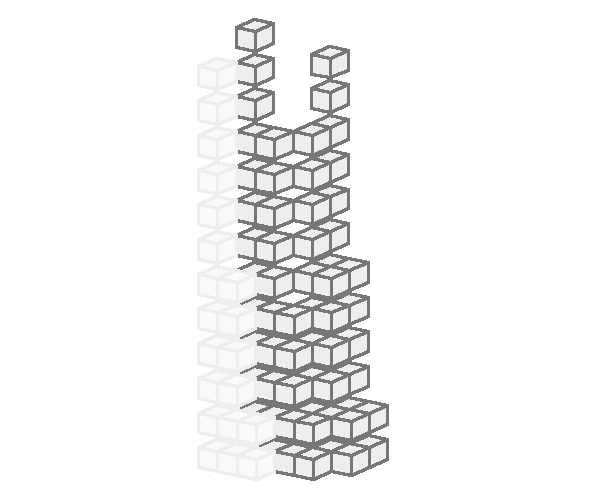} 
\end{center}
\begin{align*}
\vcenter{ \hbox{ \resizebox{2\width}{2\height}{
\ytableausetup{boxsize=0.2em}
\ydiagram[*(white)]{ 3+3, 3+3, 1+4 } 
*[*(darkgray)]{0,0,0} } } }
= \hspace{-8mm}
\vcenter{ \hbox{ \resizebox{2\width}{2\height}{
\ytableausetup{boxsize=0.2em}
\ydiagram[*(white)]{ 3+3, 3+3, 1+4 } 
*[*(darkgray)]{0,0,0} } } }
> \hspace{-8mm}
\vcenter{ \hbox{ \resizebox{2\width}{2\height}{
\ytableausetup{boxsize=0.2em}
\ydiagram[*(white)]{ 3+2, 3+2, 1+3 } 
*[*(darkgray)]{0,0,0} } } }
= \hspace{-8mm} 
\vcenter{ \hbox{ \resizebox{2\width}{2\height}{
\ytableausetup{boxsize=0.2em}
\ydiagram[*(white)]{ 3+2, 3+2, 1+3 } 
*[*(darkgray)]{0,0,0} } } }
= \hspace{-8mm} 
\vcenter{ \hbox{ \resizebox{2\width}{2\height}{
\ytableausetup{boxsize=0.2em}
\ydiagram[*(white)]{ 3+2, 3+2, 1+3 } 
*[*(darkgray)]{0,0,0} } } }
> \hspace{-8mm}
\vcenter{ \hbox{ \resizebox{2\width}{2\height}{
\ytableausetup{boxsize=0.2em}
\ydiagram[*(white)]{ 3+2, 3+2, 1+2 } 
*[*(darkgray)]{0,0,0} } } }
> \hspace{-8mm}
\vcenter{ \hbox{ \resizebox{2\width}{2\height}{
\ytableausetup{boxsize=0.2em}
\ydiagram[*(white)]{ 3+1, 3+1, 1+2 } 
*[*(darkgray)]{0,0,0} } } }
= \hspace{-8mm} 
\vcenter{ \hbox{ \resizebox{2\width}{2\height}{
\ytableausetup{boxsize=0.2em}
\ydiagram[*(white)]{ 3+1, 3+1, 1+2 } 
*[*(darkgray)]{0,0,0} } } }
= \hspace{-8mm} 
\vcenter{ \hbox{ \resizebox{2\width}{2\height}{
\ytableausetup{boxsize=0.2em}
\ydiagram[*(white)]{ 3+1, 3+1, 1+2 } 
*[*(darkgray)]{0,0,0} } } }
= \hspace{-8mm} 
\vcenter{ \hbox{ \resizebox{2\width}{2\height}{
\ytableausetup{boxsize=0.2em}
\ydiagram[*(white)]{ 3+1, 3+1, 1+2 } 
*[*(darkgray)]{0,0,0} } } } 
\end{align*}
\begin{align*}
> \hspace{-5mm} 
\vcenter{ \hbox{ \resizebox{2\width}{2\height}{
\ytableausetup{boxsize=0.2em}
\ydiagram[*(white)]{ 3+1, 3+0, 1+1 } 
*[*(darkgray)]{3,3,1} } } }
= \hspace{-5mm} 
\vcenter{ \hbox{ \resizebox{2\width}{2\height}{
\ytableausetup{boxsize=0.2em}
\ydiagram[*(white)]{ 3+1, 3+0, 1+1 } 
*[*(darkgray)]{3,3,1} } } }
> \hspace{-5mm} 
\vcenter{ \hbox{ \resizebox{2\width}{2\height}{
\ytableausetup{boxsize=0.2em}
\ydiagram[*(white)]{ 3+0, 3+0, 1+1 } 
*[*(darkgray)]{3,3,1} } } }
>  \hspace{-5mm} 
\vcenter{ \hbox{ \resizebox{2\width}{2\height}{
\ytableausetup{boxsize=0.2em}
\ydiagram[*(white)]{ 3+0, 3+0, 1+0 } 
*[*(darkgray)]{3,3,1} } } }
= E
\end{align*}
Slices may be indexed from smallest to largest, 
or vice-versa.  
The choice will be clear from context.  

Slices may be regarded eiher as skew diagrams
with the same partition cut out from inside each,
or as cylindric partitions of the same profile
having only 1's as parts.
For convenience, we draw slices as skew diagrams.
For emphasis on the profile or for reference 
when there are one or more zero rows in $\Lambda$
we indicate the cutout partition on the left
with dark gray boxes.
The dark grey boxes do not contribute to the weight.

It is clear that $\mathrm{max}(\Lambda)$ is the number of nonempty slices
in the chain \eqref{ineqSliceChain}.
In particular, if $E < {}_1\Sigma$ in \eqref{ineqSliceChain},
then $\mathrm{max}(\Lambda) = m$.

It is also possible to keep track of $\partsize(\Lambda)$
using \eqref{ineqSliceChain}.
Remove the empty slices and the duplicates in \eqref{ineqSliceChain}
to obtain
\begin{align}
\label{ineqDistSliceChain}
E < {}_{j_1}\Sigma < {}_{j_2}\Sigma < \cdots < {}_{j_h}\Sigma.
\end{align}
Then, $\partsize(\Lambda) = h$.

In the running example, the distinct slices are as follows.  	
\begin{center}
\includegraphics[scale=0.3]{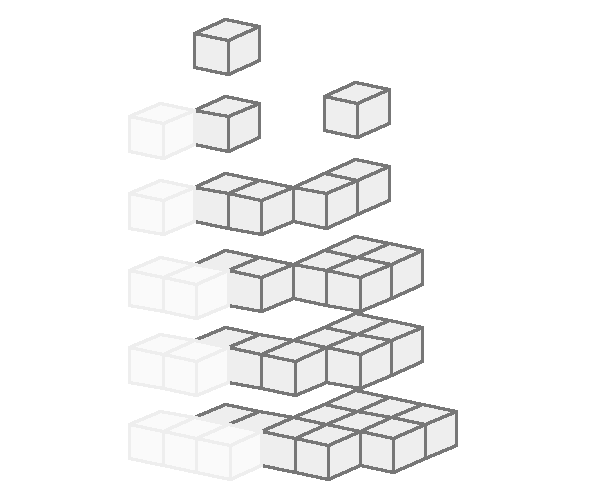} 
\end{center}
\begin{align*}
\vcenter{ \hbox{ \resizebox{2\width}{2\height}{
\ytableausetup{boxsize=0.2em}
\ydiagram[*(white)]{ 3+3, 3+3, 1+4 } 
*[*(darkgray)]{0,0,0} } } }
> \hspace{-8mm}
\vcenter{ \hbox{ \resizebox{2\width}{2\height}{
\ytableausetup{boxsize=0.2em}
\ydiagram[*(white)]{ 3+2, 3+2, 1+3 } 
*[*(darkgray)]{0,0,0} } } }
> \hspace{-8mm}
\vcenter{ \hbox{ \resizebox{2\width}{2\height}{
\ytableausetup{boxsize=0.2em}
\ydiagram[*(white)]{ 3+2, 3+2, 1+2 } 
*[*(darkgray)]{0,0,0} } } }
> \hspace{-8mm}
\vcenter{ \hbox{ \resizebox{2\width}{2\height}{
\ytableausetup{boxsize=0.2em}
\ydiagram[*(white)]{ 3+1, 3+1, 1+2 } 
*[*(darkgray)]{0,0,0} } } }
> \hspace{-5mm} 
\vcenter{ \hbox{ \resizebox{2\width}{2\height}{
\ytableausetup{boxsize=0.2em}
\ydiagram[*(white)]{ 3+1, 3+0, 1+1 } 
*[*(darkgray)]{3,3,1} } } }
> \hspace{-5mm} 
\vcenter{ \hbox{ \resizebox{2\width}{2\height}{
\ytableausetup{boxsize=0.2em}
\ydiagram[*(white)]{ 3+0, 3+0, 1+1 } 
*[*(darkgray)]{3,3,1} } } }
>  \hspace{-5mm} 
\vcenter{ \hbox{ \resizebox{2\width}{2\height}{
\ytableausetup{boxsize=0.2em}
\ydiagram[*(white)]{ 3+0, 3+0, 1+0 } 
*[*(darkgray)]{3,3,1} } } }
= E
\end{align*}

To see this inductively, start with ${}_{j_1}\Sigma$, which brings 1's only.
Then, each ${}_{j_i}\Sigma$ adds 1 to the existing parts,
and introduces one or more 1's in the intermediate cylindric partition,
incrementing the number of distinct parts from $i-1$ to $i$.
When any of the slices at hand repeats any number of times,
all existing parts greater than or equal to
some lower bound are increased by the same amount.
The parts strictly less than the same bound are intact.
This process neither can decrease nor increase the part size count.

%
%

Now, let's find the pivots in the running example.  
Using the distinct slices, 
we have the following partitions: 
\begin{align*}
\beta_0 = \hspace{-7mm}
\vcenter{ \hbox{ \resizebox{2\width}{2\height}{
\ytableausetup{boxsize=0.2em}
\ydiagram[*(white)]{ 0+3, 0+3, 0+1 } 
*[*(darkgray)]{0,0,0} } } }, \;
\beta_1 = \hspace{-7mm}
\vcenter{ \hbox{ \resizebox{2\width}{2\height}{
\ytableausetup{boxsize=0.2em}
\ydiagram[*(white)]{ 0+3, 0+3, 0+2 } 
*[*(darkgray)]{0,0,0} } } }, \;
\beta_2 = \hspace{-7mm}
\vcenter{ \hbox{ \resizebox{2\width}{2\height}{
\ytableausetup{boxsize=0.2em}
\ydiagram[*(white)]{ 0+4, 0+3, 0+2 } 
*[*(darkgray)]{0,0,0} } } }, \;
\beta_3 = \hspace{-7mm}
\vcenter{ \hbox{ \resizebox{2\width}{2\height}{
\ytableausetup{boxsize=0.2em}
\ydiagram[*(white)]{ 0+4, 0+4, 0+3 } 
*[*(darkgray)]{0,0,0} } } }, \; 
\beta_4 = \hspace{-7mm}
\vcenter{ \hbox{ \resizebox{2\width}{2\height}{
\ytableausetup{boxsize=0.2em}
\ydiagram[*(white)]{ 0+5, 0+5, 0+3 } 
*[*(darkgray)]{0,0,0} } } }, \;
\end{align*}
\begin{align*}
\beta_5 = \hspace{-7mm}
\vcenter{ \hbox{ \resizebox{2\width}{2\height}{
\ytableausetup{boxsize=0.2em}
\ydiagram[*(white)]{ 0+5, 0+5, 0+4 } 
*[*(darkgray)]{0,0,0} } } }, 
\beta_6 = \hspace{-7mm}
\vcenter{ \hbox{ \resizebox{2\width}{2\height}{
\ytableausetup{boxsize=0.2em}
\ydiagram[*(white)]{ 0+6, 0+6, 0+5 } 
*[*(darkgray)]{0,0,0} } } }, 
\beta_7 = \hspace{-7mm}
\vcenter{ \hbox{ \resizebox{2\width}{2\height}{
\ytableausetup{boxsize=0.2em}
\ydiagram[*(white)]{ 0+7, 0+7, 0+6 } 
*[*(darkgray)]{0,0,0} } } }.  
\end{align*}
They define the following skew diagrams.  
\begin{align*}
\beta_1 \backslash \beta_0 = \hspace{-7mm}
\vcenter{ \hbox{ \resizebox{2\width}{2\height}{
\ytableausetup{boxsize=0.2em}
\ydiagram[*(white)]{ 2+0, 2+0, 0+1 } 
*[*(darkgray)]{0,0,0} } } }, \;
\beta_2 \backslash \beta_1 = \hspace{-7mm}
\vcenter{ \hbox{ \resizebox{2\width}{2\height}{
\ytableausetup{boxsize=0.2em}
\ydiagram[*(white)]{ 1+1, 1+0, 0+0 } 
*[*(darkgray)]{0,0,0} } } }, \;
\beta_3 \backslash \beta_2 = \hspace{-7mm}
\vcenter{ \hbox{ \resizebox{2\width}{2\height}{
\ytableausetup{boxsize=0.2em}
\ydiagram[*(white)]{ 2+0, 1+1, 0+1 } 
*[*(darkgray)]{0,0,0} } } }, \; 
\beta_4 \backslash \beta_3 = \hspace{-7mm}
\vcenter{ \hbox{ \resizebox{2\width}{2\height}{
\ytableausetup{boxsize=0.2em}
\ydiagram[*(white)]{ 1+1, 1+1, 0+0 } 
*[*(darkgray)]{0,0,0} } } }, \;
\end{align*}
\begin{align*}
\beta_5 \backslash \beta_4 = \hspace{-7mm}
\vcenter{ \hbox{ \resizebox{2\width}{2\height}{
\ytableausetup{boxsize=0.2em}
\ydiagram[*(white)]{ 2+0, 2+0, 0+1 } 
*[*(darkgray)]{0,0,0} } } }, \;
\beta_6 \backslash \beta_5 = \hspace{-7mm}
\vcenter{ \hbox{ \resizebox{2\width}{2\height}{
\ytableausetup{boxsize=0.2em}
\ydiagram[*(white)]{ 1+1, 1+1, 0+1 } 
*[*(darkgray)]{0,0,0} } } }, 
\beta_7 \backslash \beta_6 = \hspace{-7mm}
\vcenter{ \hbox{ \resizebox{2\width}{2\height}{
\ytableausetup{boxsize=0.2em}
\ydiagram[*(white)]{ 1+1, 1+1, 0+1 } 
*[*(darkgray)]{0,0,0} } } } 
\end{align*}

Now we draw $\beta_{j-1} \backslash \beta_j$ 
and $\beta_{j} \backslash \beta_{j+1}$ together, 
for $j = $ 1, 2, 3, 4, 5, 6.  
The boxes that belong to $\beta_{j-1} \backslash \beta_j$ are drawn in red 
in each instance.  
\begin{align*}
{\color{red} \beta_1 \backslash \beta_0} 
\textrm{ and } \beta_2 \backslash \beta_1
= \hspace{-7mm} 
\vcenter{ \hbox{ \resizebox{2\width}{2\height}{
\ytableausetup{boxsize=0.2em}{
\begin{ytableau}
\none & \none & *(white)\\ 
\none & \none \\ 
*(red) 
\end{ytableau}
} } } }, \;
{\color{red} \beta_2 \backslash \beta_1} 
\textrm{ and } \beta_3 \backslash \beta_2
= \hspace{-7mm}   
\vcenter{ \hbox{ \resizebox{2\width}{2\height}{
\ytableausetup{boxsize=0.2em}{
\begin{ytableau}
\none & *(red) \\ 
\none & *(white) \\ 
*(white) 
\end{ytableau}
} } } }, \;
{\color{red} \beta_3 \backslash \beta_2} 
\textrm{ and } \beta_4 \backslash \beta_3
= \hspace{-7mm}   
\vcenter{ \hbox{ \resizebox{2\width}{2\height}{
\ytableausetup{boxsize=0.2em}{
\begin{ytableau}
\none & \none & *(white) \\ 
\none & *(red) & *(white) \\ 
*(red) 
\end{ytableau}
} } } }, 
\end{align*}
\begin{align*}
{\color{red} \beta_4 \backslash \beta_3} 
\textrm{ and } \beta_5 \backslash \beta_4
= \hspace{-7mm}   
\vcenter{ \hbox{ \resizebox{2\width}{2\height}{
\ytableausetup{boxsize=0.2em}{
\begin{ytableau}
\none & *(red) \\ 
\none & *(red) \\ 
*(white)
\end{ytableau}
} } } }, \; 
{\color{red} \beta_5 \backslash \beta_4} 
\textrm{ and } \beta_6 \backslash \beta_5
= \hspace{-7mm}   
\vcenter{ \hbox{ \resizebox{2\width}{2\height}{
\ytableausetup{boxsize=0.2em}{
\begin{ytableau}
\none & \none & *(white) \\ 
\none & \none & *(white) \\ 
*(red) & *(white)
\end{ytableau}
} } } }, \; 
{\color{red} \beta_6 \backslash \beta_5} 
\textrm{ and } \beta_7 \backslash \beta_6
= \hspace{-7mm}   
\vcenter{ \hbox{ \resizebox{2\width}{2\height}{
\ytableausetup{boxsize=0.2em}{
\begin{ytableau}
\none & *(red) & *(white) \\ 
\none & *(red) & *(white) \\ 
*(red) & *(white)
\end{ytableau}
} } } }
\end{align*}
So, $\beta_2 \backslash \beta_0$ and $\beta_4 \backslash \beta_0$ 
are the pivot slices in our chain.  

Take a slice $\Sigma$ with profile $c$
as a skew diagram $\lambda \backslash \mu$ ,
where both $\lambda$ and $\mu$ are partitions with exactly $r$ parts, allowing zeros.
Delete the last part from $\lambda$ and subtract it from all other parts of $\lambda$.
The obtained partition with exactly $(r-1)$ parts, allowing zeros,
is called the \emph{shape} of $\Sigma$.
The shape of the empty cylindric partition $E$ of the profile $c$
is called \emph{the shape of zero}.

The profile can be thought of governing the ``left ends'' of slices,
and the shapes their ``right ends''.
The convention has been representing profiles by compositions (called $c$, $d$, etc.).
We represent the shapes by partitions (called $\sigma$, $\tau$, etc.). 

It is not difficult to see that there can be at most one slice
having a given shape and weight.
For instance, the slices with weight 2
in profile $c=(1, 0, 2)$ can only have shapes $(0, 0)$, $(2, 1)$, and $(3, 3)$;
but not, say, $(1, 1)$.
In fact, we can use this observation to represent slices in a third way.
Namely, as $\textrm{weight}^{\textrm{shape}}$.
In this case, the existence of the said slice is tacitly assumed.
The distinct slices in the chain above are 
\begin{align*}
1^{( 1, 1)} < 2^{( 2, 1)} < 4^{( 1, 1)} < 6^{( 2, 2)} < 7^{( 1, 1)} < 10^{( 1, 1)}, 
\end{align*}
with $2^{( 2, 1)}$ and $6^{( 2, 2)}$ being the pivots in the chain.  
The empty slice for the profile $c=(1,0,2)$ is $ 0^{(2,2)}$.


We recall once more that the profile $c$ is fixed.
We order the slices by inclusion, and draw their
Hasse diagram~\cite{Birkhoff-Lattices}
sideways in such a way that slices which have the same weight are vertically aligned,
and slices having the same shape are horizontally aligned.
Moreover, the shapes are lexicographically ordered.
For convenience, we include the shape of zero.
This carries the profile information, as well.
For the profile $c = (1,0,2)$, the beginning of this Hasse diagram
is shown in Figure \ref{figHasse}.
\begin{figure}
\centering
\includegraphics[scale=0.1]{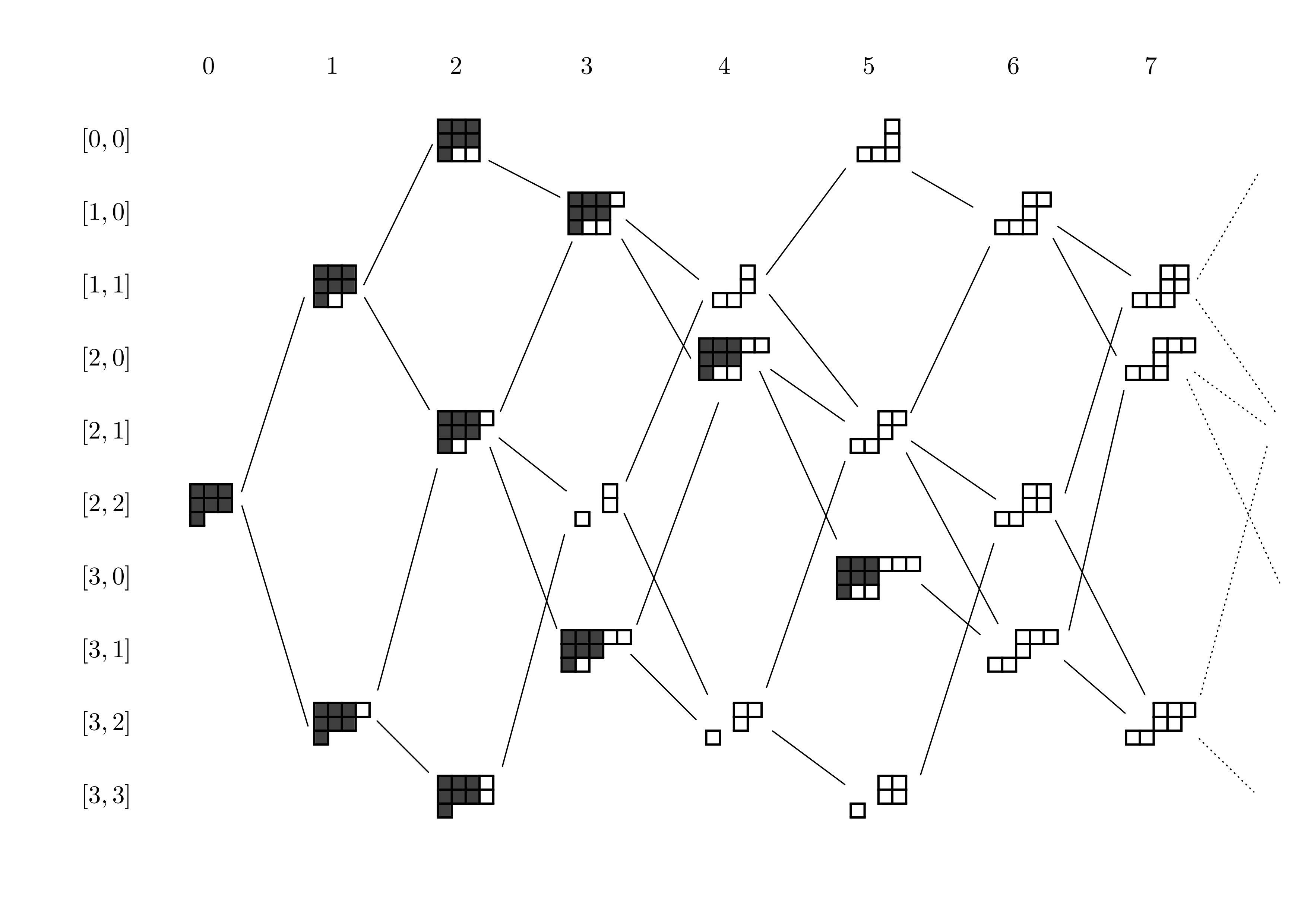}
  \caption{The Hasse diagram corresponding to the profile $c=(1,0,2)$. }
\label{figHasse}
\end{figure}

We represent a given cylindric partition in this diagram
with the multiplicities of its slices.
We replace the slices with nodes,
indicate the weights on the horizontal direction,
and the shapes on the vertical direction.
The last example will look like Figure \ref{figCylPtnInPathDiag}.
This is called the \emph{path diagram} associated with the profile $c$.  
The slice count 2 associated with $10^{( 1, 1)}$ is not visible in the figure.  
\begin{figure}
\centering
\includegraphics[scale=0.08]{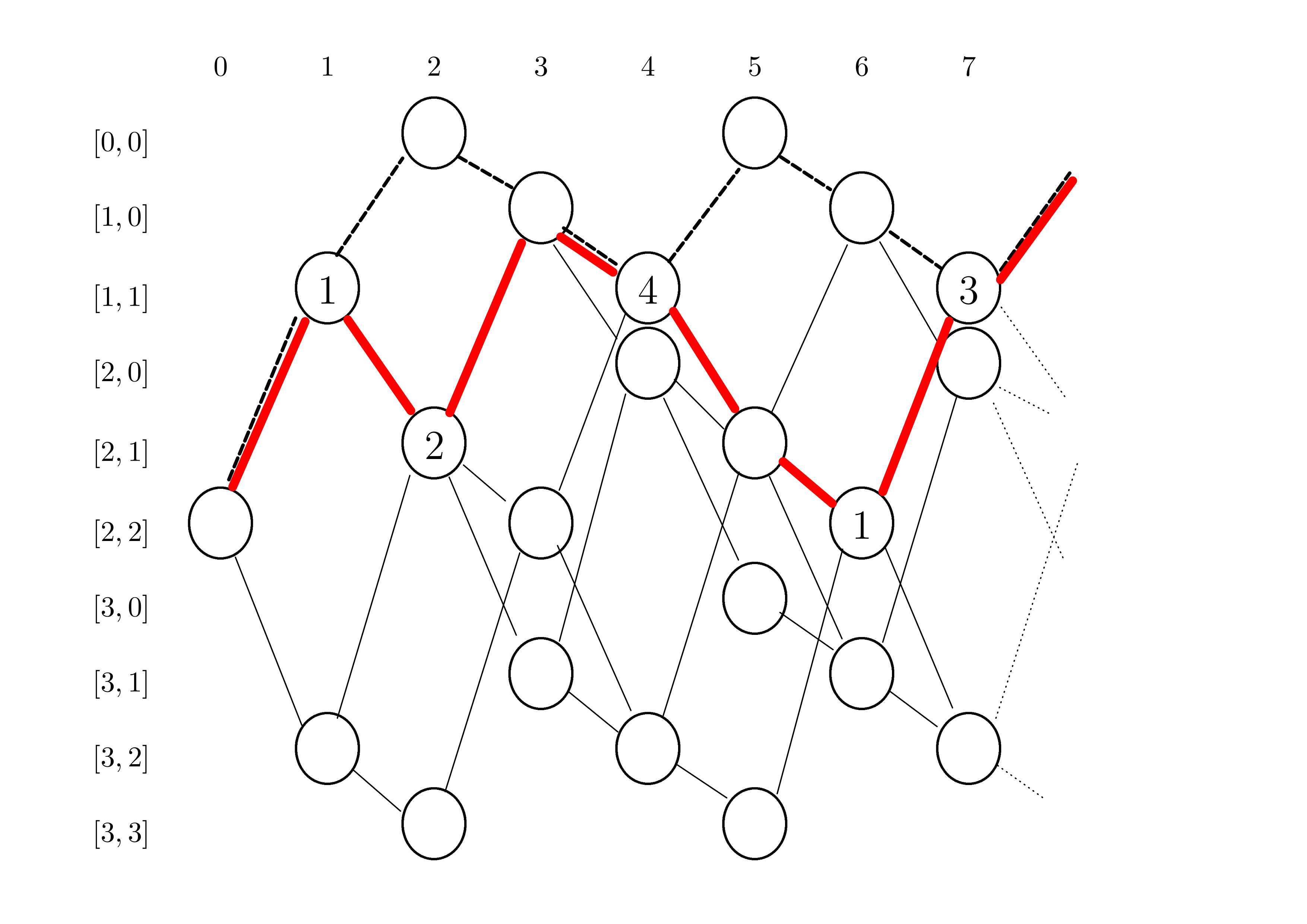}
  \caption{The path diagram corresponding to the profile $c=(1,0,2)$
    and a cylindric partition partially shown on it. }
\label{figCylPtnInPathDiag}
\end{figure}

In Figure \ref{figCylPtnInPathDiag}, 
the path between $2^{(2,1)}$ and $4^{(1,1)}$ 
goes through $3^{(1,0)}$ rather than $3^{(2,2)}$.  
This is because the path between two existing slices 
with successive weights is determined as follows.  
We obtain the bigger slice from the smaller one by adding one box at a time, 
and the order we add boxes are uniquely determined 
by tiling the rightmost column first and proceeding to the left, 
and tiling each column from the top down.  
The unique path determined by the running cylindric partition example is shown in red.  

The path starting with the empty slice 
and determined by tiling the half infinite slice to the right of it 
is called the \emph{default path}.  
The default path is shown by dashed lines in Figure \ref{figCylPtnInPathDiag}.  
Pivots govern the deviation from the default path.  
They can be thought of as providing us with necessary and sufficient data 
for navigation in a path diagram.  

Finally, consider two compositions with the same number of summands 
$c=(c_1, c_2, \ldots, c_r)$ and $d=(d_1, d_2, \ldots, d_r)$.  
Define $\Delta(d, c)$ to be the smallest non-negative $n$ such that 
the slice $n^d$ is present on the path diagram 
corresponding to the profile $c$.  

\section{A crude but general formula with some specializations}
\label{secGeneral}

\begin{proof}[Proof of Theorem \ref{thmCrudeGeneral}]
When there are no pivots,
there is a unique slice of size $n$
for any chosen and fixed positive integer $n$.
It can be found on the default path on the path diagram.
Each used slice contributes $+1$ to the maximum part,
but only the first used slice of any given and fixed size contributes to the
number of distinct slices, or to the number of part sizes in the cylindric partition.
Thus, the generating function will be
\begin{align}
\nonumber
1 + uzq^n + uz^2q^{2n} + \cdots
= 1 + u \sum_{i \geq 0} \left( zq^n \right)^i
= 1 + \frac{ uzq^n }{ 1 - zq^n}
= \frac{ 1 - (1 - u)zq^n }{ 1 - zq^n }
\end{align}
for the chosen size $n$ of the slice.
Since the presence or absence of slices of different sizes are independent,
we take the product over $n \geq 1$,
and obtain the quotient of the infinite products on the right hand side of \eqref{eqCrudeGeneral}.

On the other hand, each pivot will introduce
a new slice of a fixed size, and update the path.
Pivots, being slices themselves, contribute $+1$ to the maximum part.

However, we need a connection to the distinct slice count.
The pivot $\Pi$ itself will be thought of as contributing to the distinct slice count,
but the subsequent instances of the pivot slice
will only contribute to the maximum part.
Therefore, if we set $n = \vert \Pi \vert$,
the relevant factor in the generating function must be
\begin{align}
\label{eqCorrectFactor}
u z q^n + u z^2 q^{2n} + \cdots = \frac{ u z q^n }{ 1 - z q^n }
\end{align}
instead of
\begin{align}
\label{eqWrongFactor}
1 + u z q^n + u z^2 q^{2n} + \cdots = \frac{ 1 - (1 - u) z q^n }{ 1 - z q^n }.
\end{align}
Since \eqref{eqWrongFactor} is already a factor in the generating function,
we need to clear it by dividing.
Hence, the factor due to the pivot $\Pi$ will be
\begin{align*}
\nonumber
\frac{ \frac{ u z q^n }{ 1 - z q^n }  }{ \frac{ 1 - (1 - u) z q^n }{ 1 - z q^n }  }
= \frac{ u z q^{\vert \Pi \vert} }{ 1 - (1 - u) z q^{\vert \Pi \vert} }.
\end{align*}
The placement of a new pivot will alter the path on the path diagram,
but the contribution to the maximum part and the distinct slice count
due to slices of other sizes will not change.


The independence of the discussion for distinct pivots yields the formula.
\end{proof}

We list the application of Theorem \ref{thmCrudeGeneral}
to small profiles.
In profiles $c = (1,1)$ and $c=(2,0)$,
the pivots must have shape $(2)$.
In the former profile, they have odd weight, and in the latter even.
In both profiles,
the choice of pivots can be done on the path diagram.
One starts with an ordinary partition embedded in the path diagram
as a cylindric partition with no pivots.
Then, zero or more existing slices with shape $(0)$
are given shapes $(2)$.
For these profiles, the choices are independent.
This combinatorial process is also explained in ~\cite{K-OS}.
Therefore,
\begin{align*}
& \sum_{\Lambda \in \mathcal{P}_{(1,1)}} q^{ \vert \Lambda \vert }
z^{\mathrm{max}(\Lambda)}
u^{ \partsize(\Lambda) }
= \frac{ \left( (1 - u)zq; q \right)_\infty }{ \left( zq; q \right)_\infty }
\left[ 1 + \sum_{m \geq 1} \sum_{ \substack{ \Pi_1 < \Pi_2 < \cdots < \Pi_m \\
		\textrm{a chain of pivots} \\ \textrm{each with odd weight} } }
\prod_{j = 1}^m \frac{ u z q^{ \vert \Pi_j \vert } }{ 1 - (1 - u) z  q^{ \vert \Pi_j \vert } } \right]
\end{align*}
\begin{align*}
= \frac{ \left( (1 - u)zq; q \right)_\infty }{ \left( zq; q \right)_\infty }
\left[ 1 + \sum_{m \geq 1} \sum_{ \substack{ j_1< j_2 < \cdots < j_m \\
		\textrm{distinct odd numbers}  } }
\prod_{j = 1}^m \frac{ u z q^{ \vert \Pi_j \vert } }{ 1 - (1 - u) z  q^{ \vert \Pi_j \vert } } \right]
\end{align*}
\begin{align*}
= \frac{ \left( (1 - u)zq; q \right)_\infty }{ \left( zq; q \right)_\infty }
\prod_{m \geq 1} \left( 1 +  \frac{ u z q^{ 2m-1 } }{ 1 - (1 - u) z  q^{ 2m-1 } } \right)
= \frac{ \left( (1 - u)zq; q \right)_\infty  \left( (1 - 2u)zq; q^2 \right)_\infty }
	{ \left( zq; q \right)_\infty \left( (1 - u)zq; q^2 \right)_\infty }
\end{align*}
\begin{align*}
= \frac{ \left( (1 - u)zq^2; q^2 \right)_\infty  \left( (1 - 2u)zq; q^2 \right)_\infty }
	{ \left( zq; q \right)_\infty  }.
\end{align*}
This is \eqref{genFuncProfile11}.  
Similarly,
\begin{align}
\label{genFuncProfile20}
& \sum_{\Lambda \in \mathcal{P}_{(2,0)}} q^{ \vert \Lambda \vert }
z^{\mathrm{max}(\Lambda)}
u^{ \partsize(\Lambda) }
= \frac{ \left( (1 - u)zq; q^2 \right)_\infty  \left( (1 - 2u)zq^2; q^2 \right)_\infty }
	{ \left( zq; q \right)_\infty  }.
\end{align}

For general profiles, we can imitate the procedure in the above paragraph
keeping Theorem 4 in~\cite{K-OS} in mind.
In order to create a cylindric partition,
we start with an ordinary partition, and list the part sizes
\begin{align*}
 j_1 < j_2 < \cdots < j_m.
\end{align*}
Then, we can designate zero or more part sizes as appropriate pivots in
the path diagram corresponding to the chosen profile $c$.
After all, pivots uniquely determine a path in the path diagram.
Let's say that there are $a_c(j_1, j_2, \ldots, j_m)$ ways to choose the pivots.
The proof of Theorem \ref{thmCrudeGeneral} stipulates that
\begin{align}
 \label{eqWaysToChoosePivots}
& \sum_{\Lambda \in \mathcal{P}_{c}} q^{ \vert \Lambda \vert }
z^{\mathrm{max}(\Lambda)}
u^{ \partsize(\Lambda) }
= 1 + \sum_{m \geq 1} \sum_{ 0 < j_1 < j_2 < \cdots < j_m } a_c(j_1, j_2, \ldots, j_m)
\prod_{k = 1}^m \frac{ u z q^{ j_k } }{ 1 - z  q^{ j_k } }
\end{align}
Then, the problem transforms to calculating $a_c(j_1, j_2, \ldots, j_m)$
for a chosen profile.
It is not difficult to see that
\begin{align*}
	a_{(1,1)}(j_1, j_2, \ldots, j_m) = 2^{ \textrm{ the odd number count in } j_1, j_2, \ldots, j_m },
\end{align*}
and
\begin{align*}
	a_{(2,0)}(j_1, j_2, \ldots, j_m) = 2^{ \textrm{ the even number count in } j_1, j_2, \ldots, j_m },
\end{align*}
restoring generating functions \eqref{genFuncProfile11} and \eqref{genFuncProfile20},
respectively. 

Pivots in the profile $c=(2,1)$ are in 1-1 correspondence with 
partitions into distinct and non-consecutive parts.  
These are also known as Rogers-Ramanujan partitions~\cite{TheBlueBook}. 
Thus, $a_{(2,1)}(j_1, j_2, \ldots, j_m)$ 
is the number of ways to form a Rogers-Ramanujan partition 
using parts from $\left\{ j_1, j_2, \ldots, j_m \right\}$.  
We need to work on maximal \emph{run}s such as $n, n+1, n+2, \ldots, n+r-1$ 
in the sequence $j_1 < j_2 < \cdots < j_m$.  
By maximal we mean that $n-1 \not\in \left\{ j_1, j_2, \ldots, j_m \right\}$ , 
and $n+r \not\in \left\{ j_1, j_2, \ldots, j_m \right\}$.  
Contribution to a Rogers-Ramanujan partition 
in two disjoint maximal runs are independent.  
For a maximal run of $r$ numbers $n, n+1, n+2, \ldots, n+r-1$, 
call $g_r$ the number of ways to choose a Rogers-Ramanujan partition.  
$g_1 = 2$ since the possible partitions are the empty partition $\varepsilon$ and $n$.  
$g_2 = 3$ since the possible partitions are
the empty partition $\varepsilon$, $n$ and $n+1$.  
For $r > 2$, we see that $g_r = g_{r-1} + g_{r-2}$.  
The right hand side counts the partitions classified according to the existence 
of $n$ in the partition.  When $n$ is absent, 
we can use the remaining $r-1$ consecutive parts.  
When $n$ is present, we cannot use $n+1$, but use the remaining $r-2$ parts.  
$g_r$'s are the shifted Fibonacci numbers.  
Thus, 
\begin{align*}
& \sum_{\Lambda \in \mathcal{P}_{(2,1)}} q^{ \vert \Lambda \vert }
z^{\mathrm{max}(\Lambda)}
u^{ \partsize(\Lambda) }
= 1 + \sum_{m \geq 1} \sum_{ 0 < j_1 < j_2 < \cdots < j_m } g_{r_1} g_{r_2} \ldots g_{r_s}
\prod_{k = 1}^m \frac{ u z q^{ j_k } }{ 1 - z  q^{ j_k } }, 
\end{align*}
where the sequence $j_1 < j_2 < \ldots < j_m$ consists of exactly $s$ maximal runs 
with lengths $r_1$, $r_2$, \ldots, $r_s$ in their respective order; 
and 
\begin{align*}
  g_1 = 2, \quad g_2 = 3, 
  \quad \textrm{ and } \quad 
  g_r = g_{r-1} + g_{r-2} \textrm{ for } r > 2.  
\end{align*}
This proves \eqref{genFuncP21}.  
For instance, 
\begin{align*}
  a_{(2,1)}(2,3,4,5, 8, 10,11,12, 17, 20,21) 
  = g_4 \; g_1 \; g_3 \; g_1 \; g_2 
  = 8 \cdot 2 \cdot 5 \cdot 2 \cdot 3 
  = 480.  
\end{align*}

\section{Bijections between ordinary partitions and cylindric partitions for small profiles
  with additional constraints}
\label{secBijections}

In this section, we state identities between cylindric partitions and ordinary partitions.
The cylindric partitions have small profiles with additional constraints.
The bijective proofs are similar, so we give the proof of only one identity in detail.
There are examples for all theorems.

\begin{theorem}
\label{thmKKprfl11}
Let $\mathcal{C}^d_{(1,1)}$ denote the set of cylindric partitions
$\Lambda=(\lambda^{(1)},\lambda^{(2)})$ of profile $(1,1)$, written as
pairs
\[
P_k=[\lambda^{(2)}_k,\lambda^{(1)}_k]=[b_k,a_k],
\qquad
1\leq k\leq N,
\]
where
\[
N=\max\bigl\{\ell(\lambda^{(1)}),\ell(\lambda^{(2)})\bigr\},
\]
with zeros appended to the shorter component when necessary, such that
$a_k-b_k\in\{0,1\}$ for every $k$;
and in which every part from 1 to $\mathrm{max}(\Lambda)$ appears.
Then, $\mathcal{C}^d_{(1,1)}$ is in 1-1 correspondence
with partitions into distinct parts.
\end{theorem}

Both the theorem and the proof are easy.
This is almost the same as
taking the conjugate of a partition into distinct parts.

{\bf Example: }
Start with a partition into distinct parts.
\begin{align*}
(12,11,8,6,5,3,2)
\end{align*}
\begin{align*}
\longrightarrow
\hspace{-5mm} \hbox{ \resizebox{2\width}{2\height}{
\ytableausetup{boxsize=0.2em}
\ydiagram{ 1+6, 0+6 } } },  
\hspace{-5mm} \hbox{ \resizebox{2\width}{2\height}{
\ytableausetup{boxsize=0.2em}
\ydiagram{ 1+6, 0+5 } } }, 
\hspace{-5mm} \hbox{ \resizebox{2\width}{2\height}{
\ytableausetup{boxsize=0.2em}
\ydiagram{ 1+4, 0+4 } } },  
\hspace{-5mm} \hbox{ \resizebox{2\width}{2\height}{
\ytableausetup{boxsize=0.2em}
\ydiagram{ 1+3, 0+3 } } }, 
\hspace{-5mm} \hbox{ \resizebox{2\width}{2\height}{
\ytableausetup{boxsize=0.2em}
\ydiagram{ 1+3, 0+2 } } }, 
\hspace{-5mm} \hbox{ \resizebox{2\width}{2\height}{
\ytableausetup{boxsize=0.2em}
\ydiagram{ 1+2, 0+1 } } }, 
\hspace{-5mm} \hbox{ \resizebox{2\width}{2\height}{
\ytableausetup{boxsize=0.2em}
\ydiagram{ 1+1, 0+1 } } } 
\end{align*}
Now stack them all:
\begin{align*}
\longrightarrow
\Lambda =
\begin{array}{ccc ccc cc}
& & 7 & 6 & 5 & 3 & 2 & 2 \\
& 7 & 5 & 4 & 3 & 2 & 1 & \\
\textcolor{lightgray}{7} & \textcolor{lightgray}{6}
& \textcolor{lightgray}{5} & \textcolor{lightgray}{3}
& \textcolor{lightgray}{2} & \textcolor{lightgray}{2} & &
\end{array}
\end{align*}
Indeed, all parts from 1 to $\mathrm{max}(\Lambda) = $7 appear,
and $0 \leq a_k - b_k \leq 1$ for $k = $1, 2, \ldots, 6.

The generating function is as follows.
\begin{align*}
\sum_{ \Lambda \in \mathcal{C}^d_{(1,1)} } q^{ \vert \Lambda \vert }
z^{ \mathrm{max}(\Lambda) }
u^{ \partsize(\Lambda) }
= \left( -uzq; q^2 \right)_\infty \;
\sum_{n \geq 0} \frac{ q^{ n^2 + n } u^n z^n }{ (q^2; q^2)_n }
\end{align*}
\begin{align}
\label{genFuncKKprfl11}
= \left( -uzq; q^2 \right)_\infty \; \left( -uzq^2; q^2 \right)_\infty
= \left( -uzq; q \right)_\infty
\end{align}

We now turn to Theorem \ref{c=(1,1)bijection}.
By Theorem~4 of \cite{K-OS}, specialized to
$c=(c_1,c_2)=(1,1)$, the fixed inner partition is
\[
\gamma_0=(c_1+c_2,c_1)=(2,1).
\]
The empty slice is represented by
\[
E=\gamma_0\setminus\gamma_0.
\]
Thus, $\gamma_0$ is the fixed inner partition encoding the profile,
whereas $E$ is the empty slice. By Definition~3 of \cite{K-OS}, the
Young diagrams of the partitions in the chain are stacked with their
\emph{left and top edges aligned}. Hence, no additional column shift
is applied; the shift associated with the profile is already encoded
by $\gamma_0$. Throughout the proof, we write $\gamma_j$ in place of
the notation $\beta_j$ used in \cite{K-OS}, so as to avoid confusion
with the partition $\beta$ appearing in the statement of the theorem.
In particular,
\[
\gamma_0=\beta_0=(2,1).
\]

\begin{proof}[Proof of Theorem \ref{c=(1,1)bijection}]
Let $\mathcal{C}$ denote the set of cylindric partitions satisfying
the conditions in the statement, and let $\mathcal{P}$ denote the
set of pairs $(\mu,\beta)$ such that every positive part occurring
in $\mu$ has multiplicity exactly two and $\beta$ is a partition
into distinct odd parts.
We construct mutually inverse maps
\[
\Phi:\mathcal{C}\longrightarrow\mathcal{P}
\qquad\text{and}\qquad
\Psi:\mathcal{P}\longrightarrow\mathcal{C}.
\]
Throughout the proof, zero entries may be used as terminal padding in
the pair notation, but they are not regarded as parts of a cylindric
partition.

\medskip
\noindent
\emph{Step 1: Pair notation and the zigzag inequalities.}
Let
\[
\Lambda=(\lambda^{(1)},\lambda^{(2)})\in\mathcal{C}.
\]
We encode $\Lambda$ by the pairs
\[
P_k=[b_k,a_k],
\qquad
a_k=\lambda^{(1)}_k,
\qquad
b_k=\lambda^{(2)}_k,
\qquad
1\leq k\leq N,
\]
where zeros are appended if necessary. We also append the zero pair
\[
P_{N+1}=[0,0].
\]
By hypothesis,
\[
a_k-b_k\in\{0,1\}.
\]
Put
\[
\varepsilon_k:=a_k-b_k,
\qquad
1\leq k\leq N,
\]
and set $\varepsilon_{N+1}=0$. Thus,
\[
P_k=
\begin{cases}
[b_k,b_k],&\varepsilon_k=0,\\[2mm]
[b_k,b_k+1],&\varepsilon_k=1.
\end{cases}
\]
Since $\Lambda$ has profile $c=(1,1)$, Definition \ref{def:cylin} gives
\[
a_k\geq b_{k+1}
\qquad\text{and}\qquad
b_k\geq a_{k+1}.
\]
Together with $a_k\geq b_k$, this yields the zigzag inequalities
\[
a_k\geq b_k\geq a_{k+1},
\qquad
1\leq k\leq N.
\]
Let
\[
J(\Lambda)
=
\{k\in\{1,\ldots,N\}:\varepsilon_k=1\}
\]
be the set of indices at which $a_k=b_k+1$. Define
\[
S_k
=
\#\{j>k:j\in J(\Lambda)\}
=
\sum_{j>k}\varepsilon_j
\]
and
\[
r_k=b_k-S_k.
\]
We also put
\[
S_{N+1}=r_{N+1}=0.
\]
Since
\[
S_k-S_{k+1}=\varepsilon_{k+1},
\]
we have
\[
\begin{aligned}
r_k-r_{k+1}
&=
b_k-b_{k+1}-(S_k-S_{k+1})\\
&=
b_k-b_{k+1}-\varepsilon_{k+1}\\
&=
b_k-a_{k+1}\geq0.
\end{aligned}
\]
Thus $(r_k)$ is weakly decreasing. By the assumptions in the statement, there exists an integer $t$
with $0\leq t\leq N$ such that
\[
r_1>r_2>\cdots>r_t>0
\]
and
\[
r_k=0,
\qquad
t<k\leq N+1,
\]
where the first condition is vacuous when $t=0$.

\medskip
\noindent
\emph{Step 2: Shapes of slices for profile $(1,1)$.}
Let
\[
{}_1\Sigma\geq{}_2\Sigma\geq\cdots
\geq{}_m\Sigma>E,
\qquad
m=\max(\Lambda),
\]
be the horizontal slice decomposition of $\Lambda$, as in
Section~2.1 \cite{K-OS}. For a level $h$, write
\[
{}_h\Sigma
=
({}_h\sigma^{(1)},{}_h\sigma^{(2)})
\]
and set
\[
n_i(h)
=
l({}_h\sigma^{(i)})
=
\#\{j:\lambda^{(i)}_j\geq h\},
\qquad i=1,2.
\]
Thus,
\[
|{}_h\Sigma|=n_1(h)+n_2(h).
\]
The profile inequalities imply
\[
|n_1(h)-n_2(h)|\leq1.
\]
Indeed, if
\[
n_2(h)\geq n_1(h)+2,
\]
then
\[
\lambda^{(2)}_{n_1(h)+2}\geq h,
\]
whereas
\[
\lambda^{(1)}_{n_1(h)+1}<h.
\]
This contradicts
\[
\lambda^{(1)}_{n_1(h)+1}
\geq
\lambda^{(2)}_{n_1(h)+2}.
\]
The reverse inequality follows symmetrically.
Moreover, since $a_k\geq b_k$ for every $k$, we have
\[
n_1(h)\geq n_2(h)
\]
at every level $h$. Hence, within the family considered here, only
the following two possibilities occur:
\[
n_1(h)=n_2(h)
\qquad\text{or}\qquad
n_1(h)=n_2(h)+1.
\]
For profile $c=(1,1)$, the partition encoding the profile in
Theorem~4 \cite{K-OS} is
\[
\gamma_0=(c_1+c_2,c_1)=(2,1).
\]
Thus the empty slice is
\[
E=\gamma_0\setminus\gamma_0,
\]
and the shape of zero, computed via Definition~8 \cite{K-OS} from $\gamma_0$, is
$(1)$.
We use the letter $\gamma$ for the outer partitions of slices in
order to avoid confusion with the partition $\beta$ appearing in
the statement of the present theorem. If a slice has row lengths
$n_1$ and $n_2$, then it is represented as
\[
\gamma\setminus\gamma_0,
\qquad
\gamma=(n_1+2,n_2+1).
\]
By Definition~8 \cite{K-OS}, its shape is
\[
((n_1+2)-(n_2+1))=(n_1-n_2+1).
\]
Consequently,
\[
n_1=n_2
\quad\Longleftrightarrow\quad
\text{shape }(1),
\]
and
\[
n_1=n_2+1
\quad\Longleftrightarrow\quad
\text{shape }(2).
\]
A shape-$(1)$ slice has even weight, whereas a shape-$(2)$ slice has
odd weight. Slices of shape $(0)$ do not occur in the family under
consideration.

\medskip
\noindent
\emph{Step 3: Identification of the pivot slices.}
We first characterize the pivots using Definition~3 \cite{K-OS}. By that definition,
the Young diagrams $\gamma_0,\gamma_1,\dots$ are stacked with their
left and top edges aligned, so all column numbers below are ordinary
Young-diagram coordinates; the cylindric shift of profile $(1,1)$ is
already encoded once and for all by the fixed reference partition
$\gamma_0=(2,1)$, and no further shift is applied at each step.
Temporarily discard the multiplicities of the appearing slices and
write them as
\[
\gamma_1\setminus\gamma_0,\,
\gamma_2\setminus\gamma_0,\,
\ldots,\,
\gamma_s\setminus\gamma_0,
\]
ordered by inclusion. As in the proof of Theorem~4 of \cite{K-OS},
starting with the zero slice, complete the gaps between successive
appearing slices by inserting the boxes of
\[
\gamma_j\setminus\gamma_{j-1},
\qquad
1\leq j\leq s,
\]
one at a time, column by column from left to right and, within each
column, from top to bottom. After reaching the largest appearing slice
$\gamma_s\setminus\gamma_0$, extend the path through the hypothetical
infinite skew diagram so that it returns to the default path as quickly
as possible, as in the proof of Theorem~4 of \cite{K-OS}. When $s=0$,
take the default path itself. The resulting infinite completed path
contains exactly one slice of each nonnegative weight. The slices immediately preceding and following an
appearing slice in this completed path need not themselves be appearing
slices of $\Lambda$; they may be intermediate slices inserted during
the completion procedure. Because the boxes are inserted column by column from left to right,
the last box inserted before an appearing slice lies in the rightmost
column of the corresponding inner skew difference, while the first box
inserted after it lies in the leftmost column of the corresponding
outer skew difference. Hence, these two boxes test precisely the pivot
condition in Definition~3 of \cite{K-OS}.

Consider first an appearing shape-$(2)$ slice of weight $2k-1$, where
$k\geq1$. Its outer partition is
\[
(k+2,k).
\]
For any slice of profile $(1,1)$, Step~2 gives
\[
|n_1-n_2|\leq1,
\]
and hence its shape
\[
(n_1-n_2+1)
\]
belongs to $\{0,1,2\}$. Moreover,
\[
n_1+n_2\equiv n_1-n_2\pmod 2.
\]
Consequently, a slice has even weight if and only if it has shape
$(1)$, whereas slices of shapes $(0)$ and $(2)$ have odd weight.
Since $2k-2$ and $2k$ are both even, the slices of these immediately
adjacent weights in the completed path must both have shape $(1)$.
Their outer partitions are
\[
(k+1,k)
\qquad\text{and}\qquad
(k+2,k+1),
\]
respectively. Thus, the relevant part of the completed chain is
\[
(k+1,k)
\subset
(k+2,k)
\subset
(k+2,k+1).
\]
The box inserted immediately before the middle slice is
\[
(1,k+2),
\]
whereas the box inserted immediately after it is
\[
(2,k+1).
\]
Therefore, the rightmost column of the skew diagram immediately inside
the middle slice is $k+2$, while the leftmost column of the skew diagram
immediately outside it is $k+1$. Since
\[
k+2>k+1,
\]
the strict inequality in Definition~3 of \cite{K-OS} is satisfied.
Hence, every appearing shape-$(2)$ slice is a pivot.

For completeness, consider a shape-$(0)$ slice of weight $2k-1$, if
such a slice occurs in a general profile-$(1,1)$ cylindric partition.
Its outer partition is
\[
(k+1,k+1).
\]
In the completed path, the relevant part of the chain is
\[
(k+1,k)
\subset
(k+1,k+1)
\subset
(k+2,k+1).
\]
The box inserted immediately before the middle slice lies in column
$k+1$, whereas the box inserted immediately after it lies in column
$k+2$. Thus, the rightmost column immediately inside the middle slice
is not strictly to the right of the leftmost column immediately outside
it. Consequently, the strict inequality in Definition~3 of
\cite{K-OS} fails, and a shape-$(0)$ slice is not a pivot.

We finally verify that an appearing shape-$(1)$ slice is not a pivot.
Let its weight be $2k$, where $k\geq1$, so that its outer partition is
\[
(k+2,k+1).
\]
In the completed path, the immediately preceding slice has weight
$2k-1$ and shape either $(0)$ or $(2)$. Accordingly, the last box
inserted before the shape-$(1)$ slice lies in column $k+2$ or $k+1$.
Thus, the rightmost column immediately inside the shape-$(1)$ slice is
at most $k+2$.

Similarly, the immediately following slice in the completed path has
weight $2k+1$ and shape either $(0)$ or $(2)$. The first box inserted
after the shape-$(1)$ slice therefore lies in column $k+2$ or $k+3$.
Hence, the leftmost column immediately outside the shape-$(1)$ slice is
at least $k+2$. The rightmost column immediately inside the slice is
therefore not strictly to the right of the leftmost column immediately
outside it. Thus, the strict inequality in Definition~3 of
\cite{K-OS} cannot hold, and a shape-$(1)$ slice is not a pivot.

It follows that, within the family considered here, an appearing slice
is a pivot if and only if it has shape $(2)$.

We now identify these pivots in terms of the pairs $P_k$. Let
$k\in J(\Lambda)$, so that
\[
P_k=[b_k,b_k+1]
\qquad\text{and}\qquad
a_k=b_k+1.
\]
Consider the horizontal slice at level $a_k$ and denote it by
\[
\Pi_k:={}_{a_k}\Sigma.
\]
Notice that $k$ indexes the pair producing the pivot, whereas $a_k$
is the level of the corresponding horizontal slice.
If $j<k$, repeated use of
\[
a_i\geq b_i\geq a_{i+1}
\]
gives
\[
a_j\geq a_k.
\]
For $j=k$, this is immediate. If $j>k$, then
\[
a_j\leq b_k<a_k.
\]
Therefore, exactly the first $k$ entries of the first component are
at least $a_k$, and hence
\[
n_1(a_k)=k.
\]
Similarly, if $j\leq k-1$, the zigzag inequalities give
\[
b_j\geq a_k,
\]
whereas
\[
b_k=a_k-1<a_k.
\]
Thus exactly the first $k-1$ entries of the second component are at
least $a_k$, and therefore
\[
n_2(a_k)=k-1.
\]
Consequently,
\[
|\Pi_k|
=
n_1(a_k)+n_2(a_k)
=
k+(k-1)
=
2k-1.
\]
Moreover,
\[
n_1(a_k)-n_2(a_k)=1,
\]
so $\Pi_k$ has shape $(2)$ and is therefore a pivot.

Conversely, suppose that an appearing slice at some level $h$ has
shape $(2)$. Then, for a unique $k$,
\[
n_1(h)=k,
\qquad
n_2(h)=k-1.
\]
It follows that
\[
a_k\geq h
\qquad\text{and}\qquad
b_k<h.
\]
Since
\[
a_k-b_k\in\{0,1\},
\]
these inequalities force
\[
a_k=h=b_k+1.
\]
Therefore,
\[
k\in J(\Lambda)
\qquad\text{and}\qquad
{}_h\Sigma=\Pi_k.
\]

Thus the pivot slices of $\Lambda$ are precisely
\[
\{\Pi_k:k\in J(\Lambda)\},
\]
and their weights are
\[
\{2k-1:k\in J(\Lambda)\}.
\]
We define
\[
\beta
=
\bigl(2k-1:k\in J(\Lambda)\bigr),
\]
with its parts written in decreasing order. Hence $\beta$ is a
partition into distinct odd parts. Moreover, in profile $(1,1)$
every pivot in this family has shape $(2)$, and a slice is uniquely
determined by its weight and shape. Therefore, the parts of $\beta$
retain the complete information carried by the pivot chain.

\medskip
\noindent
\emph{Step 4: Removal of the pivot slices.}
Following Theorem~4 \cite{K-OS}, let $\Lambda^\circ$ be the
cylindric partition obtained by summing the horizontal slices
remaining after one copy of each pivot slice
\[
\Pi_j={}_{a_j}\Sigma,
\qquad j\in J(\Lambda),
\]
has been deleted from the slice decomposition of $\Lambda$.

Fix an index $k$. If $j>k$, then the zigzag inequalities imply
\[
b_k\geq a_j.
\]
Since also $a_k\geq b_k$, both entries of $P_k$ lie at or above the
level $a_j$. Therefore, removing $\Pi_j$ subtracts $1$ from both
entries of $P_k$. There are exactly $S_k$ such indices $j$.

If $k\in J(\Lambda)$, then
\[
\Pi_k={}_{a_k}\Sigma
\]
contains the upper entry $a_k$ of $P_k$, but not the lower entry
$b_k=a_k-1$. Hence removing $\Pi_k$ subtracts one additional unit
from the upper entry only.

Finally, suppose that $j<k$. Since $j\in J(\Lambda)$,
\[
a_j=b_j+1>a_{j+1}.
\]
Together with the zigzag inequalities, this gives
\[
a_j>a_{j+1}\geq a_k\geq b_k.
\]
Thus neither entry of $P_k$ reaches the level $a_j$, and $\Pi_j$
does not affect $P_k$.

Consequently, after one copy of each pivot slice has been removed,
the pair $P_k$ becomes
\[
\begin{aligned}
\bigl[b_k-S_k,\,
a_k-S_k-\varepsilon_k\bigr]
&=
\bigl[b_k-S_k,\,
b_k+\varepsilon_k-S_k-\varepsilon_k\bigr]\\
&=
\bigl[r_k,r_k\bigr].
\end{aligned}
\]
After terminal zero pairs are omitted, we obtain
\[
\Lambda^\circ
=
\bigl(
[r_1,r_1],
[r_2,r_2],
\ldots,
[r_t,r_t]
\bigr).
\]
Equivalently, both components of $\Lambda^\circ$ are equal to the
strict partition
\[
\rho=(r_1,r_2,\ldots,r_t).
\]

Let $\nu$ denote the unrestricted partition furnished by Theorem~4 \cite{K-OS}.
We use $\nu$ here, rather than the notation $\mu$ used in that
theorem, in order to reserve $\mu$ for the partition in the statement
of the present theorem. The parts of $\nu$ are the weights of the
remaining slices, each repeated according to the multiplicity of that
slice. These are exactly the horizontal slices of $\Lambda^\circ$.

Set
\[
P
=
(r_1,r_1,r_2,r_2,\ldots,r_t,r_t).
\]
At a level $h$, the two components of $\Lambda^\circ$ each contain
$\rho'_h$ entries. Hence the horizontal slice at level $h$ has
weight
\[
2\rho'_h.
\]
On the other hand,
\[
P'_h
=
\#\{\text{parts of }P\text{ that are at least }h\}
=
2\rho'_h.
\]
It follows that
\[
\nu=P'.
\]
Define
\[
\mu:=\nu'.
\]
Then
\[
\mu
=
P
=
(r_1,r_1,r_2,r_2,\ldots,r_t,r_t).
\]
Since
\[
r_1>r_2>\cdots>r_t>0,
\]
every positive part occurring in $\mu$ occurs exactly twice.

We have therefore defined
\[
\Phi(\Lambda)=(\mu,\beta)\in\mathcal{P}.
\]

\medskip
\noindent
\emph{Step 5: Construction of the inverse map.}
Let $(\mu,\beta)\in\mathcal{P}$. Since every positive part of $\mu$
occurs exactly twice, it can be written uniquely as
\[
\mu
=
(r_1,r_1,r_2,r_2,\ldots,r_t,r_t),
\]
where
\[
r_1>r_2>\cdots>r_t>0.
\]
Extend this sequence by setting
\[
r_k=0,
\qquad k>t.
\]
Since $\beta$ is a partition into distinct odd parts, define
\[
J
=
\{k\geq1:2k-1\text{ is a part of }\beta\}
\]
and
\[
\varepsilon_k
=
\begin{cases}
1,&k\in J,\\
0,&k\notin J.
\end{cases}
\]
Put
\[
S_k
=
\sum_{j>k}\varepsilon_j,
\qquad
b_k=r_k+S_k,
\qquad
a_k=b_k+\varepsilon_k.
\]
Set
\[
N=\max\bigl(t,\max J\bigr),
\]
where $\max\varnothing=0$, and form
\[
P_k=[b_k,a_k],
\qquad
1\leq k\leq N.
\]
If $N=0$, the resulting cylindric partition is empty.

Clearly,
\[
a_k-b_k=\varepsilon_k\in\{0,1\}.
\]
Thus every pair consists either of two equal entries or of two
consecutive entries. If $\varepsilon_k=1$, then
\[
P_k=[b_k,b_k+1],
\]
so the smaller entry lies in the second component, as required.
Since
\[
S_k=S_{k+1}+\varepsilon_{k+1},
\]
we have
\[
\begin{aligned}
b_k-a_{k+1}
&=
r_k+S_k-
\bigl(r_{k+1}+S_{k+1}+\varepsilon_{k+1}\bigr)\\
&=
r_k-r_{k+1}\geq0.
\end{aligned}
\]
Together with $a_k\geq b_k$, this gives
\[
a_k\geq b_k\geq a_{k+1}.
\]
These inequalities imply that both
\[
(a_1,a_2,\ldots,a_N)
\qquad\text{and}\qquad
(b_1,b_2,\ldots,b_N)
\]
are weakly decreasing. They also give
\[
b_k\geq a_{k+1}
\]
and, by transitivity,
\[
a_k\geq b_{k+1}.
\]
Therefore,
\[
\lambda^{(1)}_k\geq\lambda^{(2)}_{k+1}
\qquad\text{and}\qquad
\lambda^{(2)}_k\geq\lambda^{(1)}_{k+1}.
\]
Thus the pairs $P_k=[b_k,a_k]$ define a cylindric partition
$\Lambda$ of profile $(1,1)$.

For this cylindric partition,
\[
J(\Lambda)=J
\]
and
\[
\#\{j>k:j\in J(\Lambda)\}
=
\sum_{j>k}\varepsilon_j
=
S_k.
\]
Consequently,
\[
b_k-S_k=r_k.
\]
Hence the positive values of the resulting sequence are
\[
r_1>r_2>\cdots>r_t>0,
\]
and all subsequent values are zero. Therefore,
\[
\Lambda\in\mathcal{C}.
\]

By Step 3, the pivot slices of $\Lambda$ are precisely
\[
\Pi_k={}_{a_k}\Sigma,
\qquad k\in J,
\]
and
\[
|\Pi_k|=2k-1.
\]
Thus the partition of pivot weights is exactly $\beta$.

Removing one copy of every pivot slice changes $P_k$ into
\[
[b_k-S_k,\,
a_k-S_k-\varepsilon_k]
=
[r_k,r_k].
\]
The cylindric partition remaining after the pivots have been removed
therefore has both components equal to
\[
(r_1,r_2,\ldots,r_t).
\]
As shown in Step 4, the unrestricted partition furnished by Theorem~4 \cite{K-OS}
is the conjugate of
\[
(r_1,r_1,r_2,r_2,\ldots,r_t,r_t)=\mu.
\]
Hence the first component of the present correspondence is exactly
$\mu$.

We define
\[
\Psi((\mu,\beta))=\Lambda.
\]

\medskip
\noindent
\emph{Step 6: Verification that the constructions are inverse.}
Starting with $\Lambda\in\mathcal{C}$, the forward construction
records
\[
\varepsilon_k=a_k-b_k,
\qquad
S_k=\sum_{j>k}\varepsilon_j,
\qquad
r_k=b_k-S_k.
\]
The inverse construction then recovers
\[
b_k=r_k+S_k
\]
and
\[
a_k=b_k+\varepsilon_k.
\]
Thus every original pair $P_k=[b_k,a_k]$ is recovered, and hence
\[
\Psi\circ\Phi
=
\operatorname{id}_{\mathcal{C}}.
\]
Conversely, starting with $(\mu,\beta)\in\mathcal{P}$, the inverse
construction defines the sequences $(r_k)$ and $(\varepsilon_k)$.
For the resulting cylindric partition, the forward construction
recovers
\[
a_k-b_k=\varepsilon_k
\]
and
\[
b_k-S_k=r_k.
\]
It therefore recovers
\[
\mu
=
(r_1,r_1,\ldots,r_t,r_t)
\]
and
\[
\beta
=
(2k-1:\varepsilon_k=1).
\]
Hence
\[
\Phi\circ\Psi
=
\operatorname{id}_{\mathcal{P}}.
\]
The empty cylindric partition corresponds to
$(\varnothing,\varnothing)$, and conversely, so the empty case is
included. Therefore, $\Phi$ is a bijection from $\mathcal{C}$ to
$\mathcal{P}$.

Finally, by the weight statement in Theorem~4 \cite{K-OS},
\[
|\Lambda|
=
|\nu|
+
\sum_{k\in J(\Lambda)}|\Pi_k|.
\]
Since
\[
|\Pi_k|=2k-1,
\]
we obtain
\[
|\Lambda|
=
|\nu|
+
\sum_{k\in J(\Lambda)}(2k-1)
=
|\nu|+|\beta|.
\]
Conjugation preserves weight, so
\[
|\mu|=|\nu'|=|\nu|.
\]
Consequently,
\[
|\Lambda|=|\mu|+|\beta|.
\]
Thus the bijection is weight-preserving.
\end{proof}

\begin{ex}\label{ex:c=(1,1)bijection2}
We illustrate Theorem~\ref{c=(1,1)bijection}. Let $N=6$ and
\[
a=(a_1,\dots,a_6)=(9,7,5,2,1,1),
\qquad
b=(b_1,\dots,b_6)=(9,6,4,2,1,0),
\]
so that $\lambda^{(1)}=(9,7,5,2,1,1)$, $\lambda^{(2)}=(9,6,4,2,1)$ and

\[
\Lambda=
\begin{array}{rrrrrrr}
& 9 & 7 & 5 & 2 & 1 & 1\\
9 & 6 & 4 & 2 & 1 & &
\end{array}.
\]

One checks directly that
\[
a_k-b_k\in\{0,1\}\qquad(k=1,\dots,6),
\]
and that the zigzag inequalities hold:
\[
a_1\ge b_1\ge a_2\ge b_2\ge a_3\ge b_3\ge a_4\ge b_4\ge a_5\ge b_5\ge a_6\ge b_6,
\]
\[
9\ge9\ge7\ge6\ge5\ge4\ge2\ge2\ge1\ge1\ge1\ge0.
\]

\smallskip
\noindent\emph{Step 1: We determine the values $\varepsilon_k$, $J(\Lambda)$, $S_k$, $r_k$, $t$.}
\[
\varepsilon=(\varepsilon_1,\dots,\varepsilon_6)=(0,1,1,0,0,1),
\qquad
J(\Lambda)=\{2,3,6\}.
\]
\[
S_1=3,\ S_2=2,\ S_3=1,\ S_4=1,\ S_5=1,\ S_6=0,
\]
Then, using $r_k=b_k-S_k$ with $b=(9,6,4,2,1,0)$:
\[
r_1=b_1-S_1=9-3=6,
\qquad
r_2=b_2-S_2=6-2=4,
\qquad
r_3=b_3-S_3=4-1=3,
\]
\[
r_4=b_4-S_4=2-1=1,
\qquad
r_5=b_5-S_5=1-1=0,
\qquad
r_6=b_6-S_6=0-0=0.
\]
Since $r_1=6>r_2=4>r_3=3>r_4=1>0=r_5=r_6$, the hypothesis of Theorem~\ref{c=(1,1)bijection} holds with $t=4$.

\smallskip
\noindent\emph{Step 2:}
The indices at which a new slice actually occurs are the last index of each run of equal parts in $\lambda^{(1)}=(9,7,5,2,1,1)$, namely $k=1,2,3,4,6$; the index $k=5$ is excluded, since $a_5=a_6=1$ means $\Pi_5=\Pi_6$ is the very same slice already recorded at $k=6$. Computing $n_1(a_k),n_2(a_k)$ at each of these indices:
\[
\begin{array}{c|c|c|c|c|c|c}
k & a_k & n_1(a_k) & n_2(a_k) & |\Pi_k| & \text{shape} & \text{pivot?}\\\hline
1 & 9 & 1 & 1 & 2  & 1 & \text{no}\\
2 & 7 & 2 & 1 & 3=2(2)-1  & 2 & \text{yes}\\
3 & 5 & 3 & 2 & 5=2(3)-1  & 2 & \text{yes}\\
4 & 2 & 4 & 4 & 8  & 1 & \text{no}\\
6 & 1 & 6 & 5 & 11=2(6)-1 & 2 & \text{yes}
\end{array}
\]
Shape is $(1)$ at the non-pivots $k=1,4$ and $(2)$ at the pivots $k=2,3,6$, with $|\Pi_k|=2k-1$ precisely at the latter. Consequently,
\[
\beta=(11,5,3),
\]
a partition into distinct odd parts.

\smallskip
\noindent\emph{Step 3: removing the pivot slices.}
Using $P_k\mapsto[\,b_k-S_k,\ a_k-S_k-\varepsilon_k\,]$,
\[
P_1=[9,9]\mapsto[6,6],\quad P_2=[6,7]\mapsto[4,4],\quad P_3=[4,5]\mapsto[3,3],
\]
\[
P_4=[2,2]\mapsto[1,1],\quad P_5=[1,1]\mapsto[0,0],\quad P_6=[0,1]\mapsto[0,0].
\]
After discarding the terminal zero pairs, both components of $\Lambda^\circ$ equal the strict partition
\[
\rho=(r_1,r_2,r_3,r_4)=(6,4,3,1).
\]

\smallskip
\noindent\emph{Step 4: the partition $\mu$.}
\[
\mu=(r_1,r_1,r_2,r_2,r_3,r_3,r_4,r_4)=(6,6,4,4,3,3,1,1),
\]
in which every positive part occurs exactly twice.

\smallskip
\noindent\emph{Step 5: weight preservation.}
\[
|\Lambda|=(9+7+5+2+1+1)+(9+6+4+2+1+0)=25+22=47,
\]
\[
|\mu|=6+6+4+4+3+3+1+1=28,\qquad |\beta|=11+5+3=19,
\]
and indeed $|\mu|+|\beta|=28+19=47=|\Lambda|$.

\smallskip
\noindent\emph{Step 6: recovering $\Lambda$ from $(\mu,\beta)$ via $\Psi$.}
Starting from $(\mu,\beta)=\bigl((6,6,4,4,3,3,1,1),(11,5,3)\bigr)$, we recover $\rho=(6,4,3,1)$, hence $t=4$. From $\beta=(11,5,3)$,
\[
J=\{k\ge1:2k-1\in\beta\}=\{2,3,6\},
\]
so $\varepsilon=(0,1,1,0,0,1)$ on $\{1,\dots,6\}$ (with $N=\max(t,\max J)=6$). Then $S_1=3$, $S_2=2$, $S_3=1$, $S_4=1$, $S_5=1$, $S_6=0$, and using $b_k=r_k+S_k$:
\[
b_1=6+3=9,\quad b_2=4+2=6,\quad b_3=3+1=4,\] \[b_4=1+1=2, \quad b_5=0+1=1,\quad b_6=0+0=0,
\]
so
\[
(b_1,b_2,b_3,b_4,b_5,b_6)=(9,6,4,2,1,0).
\]
Then, using $a_k=b_k+\varepsilon_k$ with $\varepsilon=(0,1,1,0,0,1)$:
\[
a_1=9+0=9,\quad a_2=6+1=7,\quad a_3=4+1=5,\]
\[a_4=2+0=2, \quad  a_5=1+0=1,\quad a_6=0+1=1,
\]
so
\[
(a_1,a_2,a_3,a_4,a_5,a_6)=(9,7,5,2,1,1),
\]
which gives  the cylindric partition 
\[
\Lambda=
\begin{array}{rrrrrrr}
& 9 & 7 & 5 & 2 & 1 & 1\\
9 & 6 & 4 & 2 & 1 & &
\end{array}
\]
we started with. This confirms $\Psi(\Phi(\Lambda))=\Lambda$ and, symmetrically, $\Phi(\Psi((\mu,\beta)))=(\mu,\beta)$.
\end{ex}

With $\mathcal{C}$ as in Theorem \ref{c=(1,1)bijection},
\begin{align}
\label{genFuncC11Bijection}
\sum_{ \Lambda \in \mathcal{C} } q^{ \vert \Lambda \vert }
z^{ \mathrm{max}(\Lambda) }
u^{ \partsize(\Lambda) }
= \left( -uzq; q^2 \right)_\infty \;
\sum_{n \geq 0} \frac{ q^{ n^2 + n } u^n z^n }{ (zq^2; q^2)_n }.
\end{align}
This is similar to \eqref{genFuncKKprfl11},
but unless $z = 1$, it is not an infinite product.

\begin{cor}
There is a bijection between the set of cylindric partitions given in Theorem \ref{c=(1,1)bijection} and the set of ordinary partitions into distinct parts.
\end{cor}

\begin{proof}
The set of cylindric paritions given in Theorem \ref{c=(1,1)bijection} has a $1-1$ correspondence with the set of partition pairs $(\mu,\beta)$, where $\mu$ is a partition in which each parts appears exactly twice and $\beta$ is a partition into distinct odd parts. The generating function of the partition pairs $(\mu,\beta)$ is given by $(-q^2;q^2)_\infty(-q;q^2)_\infty=(-q;q)_\infty$. The generating function on the right-hand side is the generating function of ordinary partitions into distinct parts. This completes the proof.
\end{proof}

Before we state the last sample theorem,
we need a definition.

\begin{defn}
\label{defCorrectedSeq}
Let

\begin{align}
\nonumber
\Lambda = \begin{array}{ccc ccc ccc}
& & a_0 & a_1 & a_2 & \ldots & a_{r-1} & a_r \\
& &b_1 & b_2  & \ldots & b_{r-1} & b_{r} \\
\textcolor{lightgray}{a_0} & \textcolor{lightgray}{a_1} & \textcolor{lightgray}{a_2} & \textcolor{lightgray}{a_3}
& \textcolor{lightgray}{\ldots} & \textcolor{lightgray}{a_r} &
\end{array}
\end{align}
be a cylindric partition of profile $c=(2,0)$, where the pairs are given by the sequence $[b_1,a_1], [b_2,a_2], [b_3,a_3], \ldots, [b_r,a_r]$.	We note that the largest part of the
cylindric partition $\Lambda$, namely, $a_0$ is not contained in any pairs. We consider it as a single
part.
We define the set of \textit{active pairs} by
\[
S(\Lambda)=\{\,k : a_k>b_k\,\}.
\]
In other words, a pair is called \textit{active}, whenever its second entry is strictly larger than its first entry. For each $k$, we define
\[
N_k=\#\{\,j>k : j\in S(\Lambda)\,\},
\]
that is the number of active pairs to the right of the $k$th pair.
Now, we define the \emph{corrected sequence} $y_0, y_1, \ldots, y_{2r}$ by setting \[
y_0=a_0-|S(\Lambda)|.
\]
For each pair $[b_k,a_k]$,

\begin{enumerate}
\item If $b_k \geq a_k$, define
\[
y_{2k-1}=b_k-N_k,\qquad
y_{2k}=a_k-N_k.
\]

\item If $a_k>b_k$, define
\[
y_{2k-1}=a_k-N_k-1,\qquad
y_{2k}=b_k-N_k.
\]
\end{enumerate}

\end{defn}

\begin{theorem} \label{c=(2,0)bijection}
Call $\mathcal{CO}_{(2,0)}$ the collection of cylindric partitions $\Lambda$
with profile $(2,0)$ such that the positive terms of the corrected sequence
$y_0,y_1,\ldots,y_{2r}$ form a partition into distinct odd parts.
Equivalently, there exists an integer $s$, with $0\le s\le 2r$, such that
\[
y_0>y_1>\cdots>y_s>0,
\]
each positive term is odd, that is,
\[
y_i\equiv1\pmod2,\qquad 0\le i\le s,
\]
and
\[
y_{s+1}=y_{s+2}=\cdots=y_{2r}=0.
\]

Then, $\mathcal{CO}_{(2,0)}$ is in 1-1 correspondence with
the collection of pairs of partitions $(\mu, \beta)$,
where  $\mu$ is an ordinary partition into distinct odd parts
and $\beta$ is an ordinary partition into distinct even parts.
\end{theorem}

Since the proof is highly technical, we omit the lengthy details. An illustrative example of the theorem is provided below.

\begin{ex}
{\upshape Let $\mu=(9, 7, 3, 1)$ and $\beta=(8, 6, 2)$. By using the bijection mentioned in Theorem 7 of \cite{K-OS-2023}, the cylindric partition of profile $c=(2,0)$ which corresponds to the pair $(\mu, \beta)$ is

\[\Lambda=\begin{array}{ccccccc}
& & 12 & 10 & 2 & 2 & 1\\
& & 5 & 3 & 1 & & \\
\textcolor{lightgray}{12} & \textcolor{lightgray}{10} & \textcolor{lightgray}{2} & \textcolor{lightgray}{2} & \textcolor{lightgray}{1} & &
\end{array},
\]

Thus $a_0=12$, and the associated pairs are
\[
[b_1,a_1]=[5,10],\qquad
[b_2,a_2]=[3,2],\qquad
[b_3,a_3]=[1,2],\qquad
[b_4,a_4]=[0,1].
\]
Comparing each pair, we find $a_1=10>b_1=5$, $a_2=2<b_2=3$, $a_3=2>b_3=1$, and $a_4=1>b_4=0$. Hence the active set is
\[
S(\Lambda)=\{1,3,4\},\qquad |S(\Lambda)|=3.
\]
For each $k$, $N_k$ counts the active pairs strictly to its right:
\[
N_1=\#\{3,4\}=2,\qquad N_2=\#\{3,4\}=2,\qquad N_3=\#\{4\}=1,\qquad N_4=0.
\]
The corrected term coming from $a_0$ is
\[
y_0=a_0-|S(\Lambda)|=12-3=9.
\]
For the first pair (active, since $a_1>b_1$),
\[
y_1=a_1-N_1-1=10-2-1=7,\qquad y_2=b_1-N_1=5-2=3.
\]
For the second pair (not active, since $b_2\ge a_2$),
\[
y_3=b_2-N_2=3-2=1,\qquad y_4=a_2-N_2=2-2=0.
\]
For the third pair (active, since $a_3>b_3$),
\[
y_5=a_3-N_3-1=2-1-1=0,\qquad y_6=b_3-N_3=1-1=0.
\]
For the fourth pair (active, since $a_4>b_4$),
\[
y_7=a_4-N_4-1=1-0-1=0,\qquad y_8=b_4-N_4=0-0=0.
\]
Hence the corrected sequence is
\[
(y_0,y_1,\ldots,y_8)=(9,7,3,1,0,0,0,0,0).
\]
The positive terms $9>7>3>1>0$ are strictly decreasing and odd, and all remaining terms $y_4,\ldots,y_8$ vanish, so the characterization of Theorem~\ref{c=(2,0)bijection} holds with $s=3$.

Suppose we are given, for a cylindric partition $\Lambda$ of profile $c=(2,0)$ with $r=4$ pairs, the corrected sequence
\[
(y_0,y_1,\ldots,y_8)=(9,7,3,1,0,0,0,0,0),
\]
together with the active set $S(\Lambda)=\{1,3,4\}$ (equivalently, $N_1=2$, $N_2=2$, $N_3=1$, $N_4=0$, and $|S(\Lambda)|=3$). Since $y_0>y_1>y_2>y_3>0$ are strictly decreasing odd numbers and $y_4=\cdots=y_8=0$, the characterization of Theorem~\ref{c=(2,0)bijection} holds with $s=3$. We now invert the defining formulas of the theorem to recover $\Lambda$ itself.

From $y_0=a_0-|S(\Lambda)|$,
\[
a_0=y_0+|S(\Lambda)|=9+3=12.
\]
Pair $1$ is active ($1\in S(\Lambda)$), so $y_1=a_1-N_1-1$ and $y_2=b_1-N_1$ give
\[
a_1=y_1+N_1+1=7+2+1=10,\qquad b_1=y_2+N_1=3+2=5.
\]
Pair $2$ is not active ($2\notin S(\Lambda)$), so $y_3=b_2-N_2$ and $y_4=a_2-N_2$ give
\[
b_2=y_3+N_2=1+2=3,\qquad a_2=y_4+N_2=0+2=2.
\]
Pair $3$ is active, so $y_5=a_3-N_3-1$ and $y_6=b_3-N_3$ give
\[
a_3=y_5+N_3+1=0+1+1=2,\qquad b_3=y_6+N_3=0+1=1.
\]
Pair $4$ is active, so $y_7=a_4-N_4-1$ and $y_8=b_4-N_4$ give
\[
a_4=y_7+N_4+1=0+0+1=1,\qquad b_4=y_8+N_4=0+0=0.
\]

Assembling these values yields
\[
\Lambda=
\begin{array}{ccccccc}
& & 12 & 10 & 2 & 2 & 1\\
& & 5 & 3 & 1 & & \\
\textcolor{lightgray}{12} & \textcolor{lightgray}{10} & \textcolor{lightgray}{2} & \textcolor{lightgray}{2} & \textcolor{lightgray}{1} & &
\end{array},
\]
which one verifies is a valid cylindric partition of profile $c=(2,0)$ (the cross-inequalities between $a_0$ and pair $1$, and between consecutive pairs, all hold), and whose pairs $[b_1,a_1]=[5,10]$, $[b_2,a_2]=[3,2]$, $[b_3,a_3]=[1,2]$, $[b_4,a_4]=[0,1]$ indeed satisfy $a_k>b_k$ exactly for $k\in\{1,3,4\}$, so its active set is $S(\Lambda)=\{1,3,4\}$ as required.

By Theorem~\ref{c=(2,0)bijection}, since the corrected sequence of $\Lambda$ has the stated form, $\Lambda$ corresponds, under the bijection of Theorem~7 in \cite{K-OS-2023}, to the pair $(\mu,\beta)$ with
\[
\mu=(y_0,y_1,y_2,y_3)=(9,7,3,1),\qquad
\beta=(8,6,2).
\]
This is exactly the pair $(\mu,\beta)$ we started with, and $\Lambda$ is exactly the cylindric partition constructed from it earlier: the sequence, together with the active set it presupposes, determines $\Lambda$ uniquely and recovers the correspondence in both directions. }
\end{ex}

\begin{cor}
There is a bijection between the set of cylindric partitions whose characterization is given in Theorem \ref{c=(2,0)bijection} and the set of ordinary partitions into distinct parts.
\end{cor}

\begin{proof}
The cylindric partitions in Theorem \ref{c=(2,0)bijection} correspond to the pair of partitions $(\mu,\beta)$, where $\mu$ is a partition into distinct odd parts and $\beta$ is a partition into distinct even parts, as stated in the Theorem \ref{c=(2,0)bijection}. The generating function of those pairs of partitions $(\mu,\beta)$ is given by $(-q;q^2)_\infty(-q^2;q^2)_\infty=(-q;q)_\infty$. The generating function on the right-hand side is the generating function of ordinary partitions into distinct parts. This completes the proof.
\end{proof}

\section{Cylindric partitions into distinct parts}
\label{secDist}

The part size count in cylindric parts into distinct parts
comes almost for free~\cite{K-OS}.

\begin{theorem}
\label{thmDistParts}
Let $\mathcal{D}_c(m, n)$ denote the number of cylindric partitions of $n$
with profile $c$ into $m$ distinct parts.
Set
\begin{align}
\nonumber
D_c(u; q) = \sum_{m, n \geq 0} \mathcal{D}_c(m, n) u^m q^n.
\end{align}
Then, there exist a rational function $rg_c(u; q)$ which depends on the profile $c$,
$k \in \mathbb{N}$ and doubly indexed complex numbers
$\alpha_{\cdots}^{\cdots}$, $\beta_{\cdots}^{\cdots}$
such that $D_c(u; q)$ is the following finite linear combination of
infinite products and infinite products multiplied by Lambert series,
up to a correction by the rational function $rg_c(u; q)$.
\begin{align}
\nonumber
D_c(u; q) = rg_c(u; q) + & \sum_{i = 1}^{B_0} \alpha_i^0 \; ( - \beta_i^0 u q ; q)_\infty \\
\nonumber
& + \sum_{i = 1}^{B_1} \alpha_i^1 \;
\left[ \left( z \frac{\mathrm{d}}{\mathrm{d}z} \right)
( - \beta_i^1 u z q ; q)_\infty \right]_{z = 1} \\
\nonumber
& \vdots \\
\nonumber
& + \sum_{i = 1}^{B_k} \alpha_i^k \;
\left[ \left( z \frac{\mathrm{d}}{\mathrm{d}z} \right)^k
( - \beta_i^k u z q ; q)_\infty \right]_{z = 1}
\end{align}
\end{theorem}

\begin{proof}
The proof is  the proof of Theorem 5 in~\cite{K-OS}
with an additional statistic being tracked.
We will only give the sketch of the proof.

For $n \geq 1$,
let $\Lambda$ be a cylindric partition into $(n-1)$ distinct parts.
Call the profile of $\Lambda$ $c$,
and the shape of its largest slice ${}_{n-1}\Sigma$, namely the the bottommost slice, $d$.
The procedure to introduce an $n$th part is as follows.


Pick a slice  ${}_{n}\Sigma$ with weight $n$, profile $c$ and shape $f$. ${}_{n}\Sigma$ must necessarily be a neighbor of  ${}_{n-1}\Sigma$ in the path diagram.
Place ${}_{n}\Sigma$ under ${}_{n-1}\Sigma$,
so each part in $\Lambda$ is increased by one, and a new part 1 is introduced.
This calls for a factor of $uq^n$ in the terms of the generating function
for cylindric partitions into $n$ or more distinct parts.

Then, we can keep adding the same ${}_{n}\Sigma$ under its first occurrence.
Each copy will increment each part of the cylindric partition at hand by 1,
it will add $n$ to the weight,
but none of them will add a new distinct part.
Thus, their contribution to the generating function will be
\begin{align}
\nonumber
1 + q^n + q^{2n} + \cdots = \frac{1}{1 - q^n}.
\end{align}
We clearly have the empty slice for the cylindric partition
into zero distinct parts.
Let's also recall that the distinct  slices
$E$ $< {}_{1}\Sigma$ $< {}_{2}\Sigma$ $< \cdots$ $< {}_{n}\Sigma$
form a path of length $n$ in the path diagram
starting with the empty slice $E$, where $\left\vert {}_{j}\Sigma \right\vert = j$.
Let's call the number of such paths $a_n$.
Mathematical induction yields the term of the generating function
of cylindric partitions into $n$ distinct parts as
\begin{align}
\nonumber
a_n \frac{u^n q^{\binom{n+1}{2}}}{ (q; q)_n },
\end{align}
Thus,
\begin{align}
\label{eqGenFuncDistPartsInter}
D_c(u; q) = \sum_{n\geq0} a_n \frac{u^n q^{\binom{n+1}{2}}}{ (q; q)_n }.
\end{align}
The rest of the proof is computational details,
which can be found in~\cite[Section 4]{K-OS}.
\end{proof}

Examples from~\cite{K-OS} immediately carry over.  
\begin{align}
 D_{(2,0,0)}(u; q) =
 \left( \frac{1}{2} + \frac{\sqrt{5}}{10} \right) 
  \left( - \frac{(1 + \sqrt{5})}{2} u q ; q \right)_\infty
 + \left( \frac{1}{2} - \frac{\sqrt{5}}{10} \right) 
  \left( - \frac{(1 - \sqrt{5})}{2} u q ; q \right)_\infty 
\end{align}
\begin{align*}
	D_{(1,1,1)}(u; q) = -\frac{1}{2} + \frac{3}{2} (- 2 u q; q)_\infty
\end{align*}
\begin{align*}
  D_{(4,0)}(u; q)
  = \frac{1}{3}
  + \left( \frac{2 + \sqrt{3}}{6} \right) ( \sqrt{3} \; u q; q)_\infty
  + \left( \frac{2 - \sqrt{3}}{6} \right) ( - \sqrt{3}\; u q; q)_\infty
\end{align*}

It is possible to use the above proof to write
\begin{align}
\label{eqGenFuncDistPartsMax}
D_c(z, u; q) = \sum_{m, n, b \geq 0} \mathcal{D}_c(b, m, n) z^b u^m q^n
= \sum_{n\geq0} a_n \frac{u^n z^n q^{\binom{n+1}{2}}}{ (zq; q)_n }.
\end{align}
where $\mathcal{D}_c(b, m, n)$ is the number of cylindric partitions
with profile $c$ into $m$ distinct parts and the maximum part is $b$.
The $a_n$'s are the number of paths of length $n$ on the
path diagram associated with profile $c$ starting with the empty slice.
However, we do not have an infinite product representation of \eqref{eqGenFuncDistPartsMax}.

\section{Functional equations}
\label{secFuncEqs}

Let 
\begin{align*}
  f_c(n; u) := f_c(n; u, q) 
  = \sum_{ \substack{ \Lambda \textrm{ with profile } c \\ 
      \mathrm{max}(\Lambda) \leq n} }
    q^{ \vert \Lambda \vert } u^{ \partsize(\Lambda) }.  
\end{align*}
In other words, 
$f_c(n; u)$ generates cylindric partitions with profile $c$ 
whose parts are at most $n$.  
When we want to incorporate $z^{\mathrm{max}(\Lambda)}$, 
we use 
\begin{align*}
  f_c(n; u) - f_c(n-1; u)
  = \sum_{ \substack{ \Lambda \textrm{ with profile } c \\ 
      \mathrm{max}(\Lambda) = n} }
    q^{ \vert \Lambda \vert } u^{ \partsize(\Lambda) }.  
\end{align*}
Then, we can multply either side by $z^n$ 
and write various other generating functions 
along with their functional equations, 
as we will be shown below.  
It follows that $f_c(0; u) = 1$, 
and it is understood that $f_c(n; u) = 0$ for $n < 0$.  

\begin{theorem}
\label{thmFuncEqs}
For an arbitrary but fixed profile $c = (c_1, c_2, \ldots, c_r)$
with  rank $r$ and level $\ell = c_1 + c_2 + \cdots + c_r$,
\begin{align}
\label{eqFuncEqInitVals}
f_c(0; u) = 1,
\end{align}
and for $n \geq 1$
\begin{align}
\label{eqFuncEq}
f_c(n; u) = f_c(n-1; u) + \frac{ u q^{nr} }{ ( 1 - q^{nr} ) } f_c(n-1; u)
+ \sum_{d \neq c} \frac{ u q^{ n \Delta(c,d) } }{ ( 1 - q^{nr} ) } f_d(n-1; u),
\end{align}
where the sum is over all profiles $d$ except for $c$
with rank $r$ and level $\ell$.
\end{theorem}

\begin{proof}
This proof is an alternative point of view of the construction
in Section 4 of~\cite{K-OS}
with an additional statistic being tracked,
namely the part size count.

We consider a cylindric partition $\Lambda$
with an arbitrary but fixed profile $d$ with parts at most $n-1$.
$\Lambda$ is generated by $f_d(n-1; u)$.
If there are no parts equal to $n-1$ in $\Lambda$,
this means it contains one or more empty slices $E$ with profile $d$.
We would like to obtain a cylindric partition $M$
with profile $c$ with parts at most $n$.
There are two cases to consider.

If $d \neq c$, then we need at least $\Delta(d,c)$ $n$'s to place at the left end of $\Lambda$.
In other words, we append $\Delta(d, c)$ boxes to the left of all slices of $\Lambda$,
including the empty slices to endow them with profile $c$.
On top of it, we place a slice,
the smallest slice with profile $c$ and shape $d$.
The last added slice has weight $\Delta(d,c)$ by definition.
The introduction of a new part $n$ requires an additional factor of $u$,
because $\Lambda$ could not have had any $n$'s.
We have obtained a $M$ as described in the above paragraph,
and, we have the account of the factor $u q^{ n \Delta(c,d) }  f_d(n-1; u)$
in the sum in \eqref{eqFuncEq}.

We can either stop here,
or we can add one box on the left of each row of each slice of $M$,
adding $rn$ boxes in total.
Since $M$ already had $n$'s,
this addition does not introduce a new part.
So, no need for an additional factor of $u$.
We can repeat it as many times as we want,
explaining the $\frac{ 1 }{ (1 - q^{rn}) }$ in the sum at the end of \eqref{eqFuncEq}.
Altogether, we have the part of the proof
pertaining to the sum over profiles $d$ that are different from $c$.

If $d = c$, we first take $M$ as $\Lambda$
with an additional empty slice $E$ with profile $c$.
Because the empty slice does not alter the weight or the distinct part count,
this is practically taking $\Lambda$ generated by $f_c(n-1; u)$,
and regarding it as $M$ generated by $f_c(n; u)$.
This explains the first term on the right hand side of \eqref{eqFuncEq}.

Then, we can repeat the adding boxes,
one to the left end of each row of each slice in $M$,
including its initially empty slices,
a total of $rn$ boxes in each round.
The first round introduces the new part $n$,
so we need an extra factor of $u$.
The subsequent rounds do not add a new part,
so we do not need any more factors of $u$.
We now have the explanation of the factor
\begin{align}
\nonumber
u q^{rn} + u q^{2rn} + \cdots = \frac{ u q^{nr} }{ ( 1 - q^{nr} ) }
\end{align}
in the second term on the right hand side of \eqref{eqFuncEq},
and conclude the proof.
\end{proof}

{\bf Example: }
The cylindric partition 
\begin{align*}
  \begin{array}{ccc ccc ccc}
  & & & 12 & 6 & 2 &\\
  & & & 10 & 6 & 2 & \\
  & 13 & 10 & 5 & 2 & \\
  \textcolor{lightgray}{12} & \textcolor{lightgray}{6}
  & \textcolor{lightgray}{2} & 
  \end{array}
\end{align*}
is generated by $f_{(1, 0, 2)}(13, u)$.  
If we excise the squares containing the largest part 13, 
we end up with 
\begin{align*}
  \begin{array}{ccc ccc ccc}
  & & & 12 & 6 & 2 &\\
  & & & 10 & 6 & 2 & \\
  & & 10 & 5 & 2 & \\
  \textcolor{lightgray}{12} & \textcolor{lightgray}{6}
  & \textcolor{lightgray}{2} & 
  \end{array}, 
\end{align*}
a cylindric partition generated by $f_{(2, 0, 1)}(12, u)$.  
This excision is covered by the term 
\begin{align*}
  \frac{ u q^{ 13 } }{ ( 1 - q^{3 \cdot 13} ) } f_{(2, 0, 1)}(12; u).  
\end{align*}
Moreover, the denominator did not contribute this time.  
Excision of the current largest part 12 
will similarly take us through the term 
\begin{align*}
  \frac{ u q^{ 12 } }{ ( 1 - q^{12 \cdot 3} ) } f_{(1, 1, 1)}(11; u), 
\end{align*}
again with no contribution from the denominator, and then 
through the term 
\begin{align*}
  f_{(1, 1, 1)}(10; u), 
\end{align*}
because the cylindric partition 
\begin{align*}
  \begin{array}{cc ccc ccc}
  & & & 6 & 2 &\\
  & & 10 & 6 & 2 & \\
  & 10 & 5 & 2 & \\
  \textcolor{lightgray}{6}
  & \textcolor{lightgray}{2} & 
  \end{array} 
\end{align*}
is generated both by $f_{(1, 1, 1)}(11; u)$ and by $f_{(1, 1, 1)}(10; u)$.  
In fact, it is generated by $f_{(1, 1, 1)}(n; u)$ for $n \geq 10$.  
In the passage to $f_{(1, 1, 1)}(10; u)$, 
no factors of $u$ are encountered, 
because the part size count is preserved.  
After similar steps, we come by the cylindric partition 
\begin{align*}
  \begin{array}{cc cc}
  & & & 2 \\
  & & & 2 \\
  & & 2 & \\
  \textcolor{lightgray}{2} & & & 
  \end{array} 
\end{align*}
generated by $f_{(2, 0, 1)}(2; u)$.  
Excising the squares with the largest part 2 this time 
yields the empty cylindric partition with the same profile, 
and this is realized through the term 
\begin{align*}
  \frac{ u q^{3 \cdot 2} }{ ( 1 - q^{ 3 \cdot 2} ) } f_{2, 0, 1}(1; u)
\end{align*}
using the contribution $q^6$ from the denominator.

We will finally sketch some formal derivations.  
The reason for the verbosity is to correct the inaccuracy 
in equation (18) in~\cite{K-OS}.  
Let us keep in mind that $c$ and $d$ are compositions with $r$ parts, 
i.e. we are dealing with cylindric partitions 
with arbitrary but fixed rank $r$ throughout.  
The finite sums below are over all compositions $d$ 
with weight $\ell$ except $c$.  
Set
\begin{align*}
  P_c(n; u) = (q^r; q^r)_n \; f_c(n; u).  
\end{align*}
It is not straightforward, but possible 
to have a combinatorial interpretation of $P_c(n; u)$.  
We are not going to attempt it here.  
Then, \eqref{eqFuncEq} translates to 
\begin{align}
\label{eqIntermFuncEq}
  P_c(n; u) = \left( 1 - (1 - u)q^{nr} \right) P_c(n-1; u)
    + u \sum_{d \neq c} q^{ n \Delta(d, c) } P_d(n-1; u).  
\end{align}
Now define 
\begin{align*}
  G_c(u, z) = G_c(u, z; q) = \sum_{n \geq 0} \frac{ P_c(n; u) z^n }{ (q^r; q^r)_n }.  
\end{align*}
This time, \eqref{eqIntermFuncEq} translates to 
\begin{align}
\label{eqFuncEqAlmostThere}
  (1 - z)G_c(u, z) = \left( 1 - (1 - u)z q^r \right) G_c(u, z q^r) 
    + u z \sum_{ d \neq c } q^{ \Delta( d, c ) } G_d( u, z q^{ \Delta(d, c) } ).  
\end{align}
Since 
\begin{align*}
  G_c(u, z) = \sum_{ n \geq 0 } f_c(n; u) z^n, 
\end{align*}
\begin{align*}
  F_c(u, z) = \sum_{ n \geq 0 } (f_c(n; u) - f_c(n-1; u) ) z^n, 
\end{align*}
and hence 
\begin{align}
\label{eqFuncEqUltimate}
  F_c(u, z) = \frac{\left( 1 - (1 - u)z q^r \right)}{( 1 - z q^r )} F_c(u, z q^r) 
    + u z \sum_{ d \neq c } \frac{q^{ \Delta( d, c ) }}{ 1 - z q^{ \Delta( d, c ) } } 
      F_d( u, z q^{ \Delta(d, c) } ).  
\end{align}
Setting $u = 1$ corrects equation (18) in ~\cite{K-OS}.

\section{Conclusion and future work}
\label{secConc}

The concept of part sizes in integer partitions 
goes back to MacMahon~\cite{MacMahon}.  
One quest is to investigate the possibility to find formulas 
for suitably restricted cylindric partitions \emph{with designated pivots} 
similar to those in~\cite{ALL}.  

An obvious direction for future research is 
looking for bijections involving cylindric partitions with larger profiles.  
Also, one would like to see Theorem \ref{c=(2,0)bijection}
with a simpler and more natural statement.

It is possible to express $a_{c}(j_1, j_2, \ldots, j_m)$ 
described towards the end of Section \ref{secGeneral} 
for larger profiles in terms of recurrences, 
but the recurrences get messier as the rank $r$ or the level $\ell$ increases.  
An open question is the study of $a_{c}(j_1, j_2, \ldots, j_m)$ 
for arbitrary but fixed $c$.

%

\bibliographystyle{amsplain}

\end{document}